\documentclass{amsart}

\usepackage{amssymb}
\usepackage{amsmath}
\usepackage{amsthm}
\usepackage[dvips]{graphicx}
\usepackage{color}

\newtheorem{theo}{Theorem}[section]
\newtheorem{lemm}[theo]{Lemma}

\newtheorem{defi}[theo]{Definition}

\newtheorem{prop}[theo]{Proposition}
\newtheorem{rema}[theo]{Remark}

\newcommand\RR{{\mathbb R}}

\def\SS{\mathbb {S}}

\def\e{\eqno}
\def\la{\langle}
\def\ra{\rangle}

\let\a=\alpha
\let\b=\beta
\let\d=\delta
\let\e=\epsilon
\let\g=\gamma

\let\o=\omega
\let\p=\psi

\let\s=\sigma

\let\vp=\varphi

\let\D=\Delta

\newcommand\bN{{\mathbb N}}

\newcommand\bR{{\mathbb R}}

\newcommand\bS{{\mathbb S}}

\newcommand\cC{{\mathcal C}}
\newcommand\cD{{\mathcal D}}

\newcommand\cF{{\mathcal F}}

\newcommand\cS{{\mathcal S}}
\newcommand\cM{{\mathcal M}}

\newcommand\cK{{\mathcal K}}
\newcommand\cG{{\mathcal G}}
\newcommand\cX{{\mathcal X}}

\let\wt=\widetilde

\newcommand\pa{\partial}

\newcommand\re{\mathop{\rm Re}\, }

\let\dis=\displaystyle

\date{today}
\begin{document}
\title[Measure valued Solutions]
{Measure valued solutions  to the spatially\\ homogeneous
Boltzmann equation\\ without angular cutoff
}
\author{Yoshinori Morimoto }
\address{Yoshinori Morimoto, Graduate School of Human and Environmental Studies,
Kyoto University,
Kyoto, 606-8501, Japan} \email{morimoto@math.h.kyoto-u.ac.jp}
\author{Shuaikun Wang}
\address{Shuaikun Wang , Department of mathematics, City University of Hong Kong,
Hong Kong, P. R. China
} \email{shuaiwang4-c@my.cityu.edu.hk}
\author{Tong Yang}
\address{Tong Yang, Department of mathematics, City University of Hong Kong,
Hong Kong, P. R. China; \& Department of Mathematics, Shanghai Jiao Tong University, Shanghai, P. R. China
} \email{matyang@cityu.edu.hk}



\date{March 3, 2016}

\begin{abstract}
A uniform approach is introduced to study the existence of
measure valued solutions to the homogeneous Boltzmann equation
for both hard potential with finite energy, and soft potential with finite or infinite energy, by using Toscani metric. Under the
non-angular cutoff assumption on the cross-section, the solutions
obtained are shown to be in the Schwartz space in the velocity variable
as long as the initial data is not a single Dirac mass without any
extra moment condition for hard potential, and with the boundedness
on moments of any order for soft potential.
\end{abstract}
\maketitle

\section{Introduction}\label{s1}

The spatially 
 homogeneous Boltzmann equation;
\begin{equation}\label{BE}
\partial_t f(t,v) 
=Q(f, f)(t,v),
\end{equation}
describes the time evolution of the velocity distribution of a spatially
homogeneous dilute gas of particles. The most interesting and
important part of this equation is the collision operator
given on the
 right hand side that captures the change rates of the density distribution
through elastic binary collisions:
\[
Q(g, f)(v)=\int_{\RR^3}\int_{\mathbb S^{2}}B\left({v-v_*},\sigma
\right)
 \left\{g(v'_*) f(v')-g(v_*)f(v)\right\}d\sigma dv_*\,,
\]
where for $\sigma \in \SS^2$
$$
v'=\frac{v+v_*}{2}+\frac{|v-v_*|}{2}\sigma,\,\,\, v'_*
=\frac{v+v_*}{2}-\frac{|v-v_*|}{2}\sigma,\,
$$
that follow from the conservation of momentum and energy,
\[ v' + v_*' = v+ v_*, \enskip |v'|^2 + |v_*'|^2 = |v|^2 + |v_*|^2.
\]

We consider the non-negative cross section
$B$  of  the form
\begin{align}\label{cross-section}
B(|v-v_*|, \cos \theta)=\Phi (|v-v_*|) b(\cos \theta),\,\,\,\,\,
\cos \theta=\frac{v-v_*}{|v-v_*|} \, \cdot\,\sigma\, , \,\,\,
0\leq\theta\leq\frac{\pi}{2},
\end{align}
where
\begin{align}
&\Phi(|z|)=\Phi_\gamma(|z|)= |z|^{\gamma},\mbox{ for some $\gamma>-3$},   \label{cross-section1} \\
& b(\cos \theta)\theta^{2+2s}\,\rightarrow K
\mbox{ when } \theta\rightarrow 0+,
\mbox{ for } 0<s<1 \mbox{ and }K>0.\label{cross-section2}
\end{align}
For the case $\g>0$, we speak of hard potentials; For the case $-3<\g<0$, we speak of soft potentials, and for the case $\g=0$, we speak of Maxwellian molecule. 

An important example of this singular cross section is  derived from the inverse power law potential $\rho^{-r}$, $\rho$ being the distance between two interacting particles, in which $s=\frac1r\in(0,1)$ and $\g=1-4s\in(-3,1)$.
As usual, the range of $\theta$ can be restricted to $[0,\pi/2]$, by replacing $b(\cos\theta)$
by its ``symmetrized'' version
\[
[ b(\cos \theta)+b(\cos (\pi-\theta))]{\bf 1}_{0\le \theta\le \pi/2}.
\]

The equation \eqref{BE} is supplemented with an  initial datum 
\begin{equation}\label{initial}
f(0,v) = F_0,
\end{equation}
which is generally assumed to be a density of a probability measure, after normalizing time variable $t$ suitably.
We consider the Cauchy problem \eqref{BE}-\eqref{initial} for an initial datum $F_0$ with a finite second
moment, that is, $F_0 \in P_2(\RR^3)$. Here, in general,  we denote by 
$P_\alpha(\RR^3)$, $\a \in [0, \infty)$, the set of  probability measures  $F$ on $\RR^3$, such that
\[
\int_{\RR^3} |v|^\alpha dF(v) < \infty, \,
\]
and moreover when  $ \a >1 $, it requires that
\begin{align}\label{mean}
\int_{\RR^3} v_j dF(v) = 0, \enskip j =1,2,3\,,
\end{align}
which is always satisfied by the translation in $\RR^3$.  

The purpose of the present paper is
to show that,  for both hard ($\gamma >0$) and moderately soft ($0 > \gamma \ge -2$) potentials, 
the Cauchy problem  \eqref{BE}-\eqref{initial} have 
a measure valued solution in the following sense.

\begin{defi}[Measure valued solution]\label{weak-solution}
	Let $B(\cdot)$ satisfy \eqref{cross-section}-\eqref{cross-section2}
and let $F_0 \in P_\a(\RR^3)$ with $0 < \a \le 2$.
We say that $F_t\in L^\infty([0,\infty); P_\a(\bR^3))$ is a measure valued solution to the Cauchy problem 
 \eqref{BE}-\eqref{initial}
if it satisfies
for every $\psi(v)\in C_b^2(\bR^3)$,
		\begin{align}\label{definition2-2}
	\int_{\bR^3}\psi(v)dF_t=\int_{\bR^3}\psi(v)dF_0+\frac12&\int_0^t\int_{\bR^3}\int_{\bR^3}\int_{\bS^2}b(\cdot)|v-v_*|^\g (\psi(v_*')+\psi(v')\\
		&-\psi(v_*)-\psi(v))dF_\tau(v)dF_\tau(v_*)d\s d\tau.\nonumber
		\end{align}
In particular, when $\a =2$, 
the energy conservation law holds
\begin{align}\label{definition2-1}
		\forall t \ge 0,  \enskip \int|v|^2dF_t =  \int|v|^2dF_0\,,
		\end{align}
so that $F_t\in C([0,\infty); P_2(\bR^3))$.		

\end{defi}

It should be noted that the moment conservation law holds because of  \eqref{mean}, and the definition does not contain
the finite entropy condition (cf., Definition 1 of \cite{villani}). The definition implicitly requires
the finiteness of the second term of the right hand side of \eqref{definition2-2}, similar as in Definition 1.1 of \cite{lu-mouhot}.
We are now ready to state our main results.

\begin{theo}\label{main-theo}
	Let $B(\cdot)$ satisfy \eqref{cross-section}-\eqref{cross-section2} and let $-2\le \gamma \le 2$.
 
\noindent
(1)  For any initial datum  $F_0\in P_2(\bR^3)$, we have  a measure valued solution $F_t$ $\in$ 
$C([0,\infty);P_2(\bR^3))$ to the Cauchy problem \eqref{BE}-\eqref{initial}.
In the hard potential case $ \gamma >0$, for any $\ell>0$ and  any $t_0 >0$ we have a measure valued solution $F_t$ $\in$ 
$C([0,\infty);P_2(\bR^3))$ such that 
\begin{equation}\label{moment-gain}
\mbox{$\dis \int|v|^\ell dF_{t} < \infty$ for any $t \ge t_0$.}
\end{equation}
In the (moderately) soft potential case $0 > \gamma \, (\ge -2)$,  we have a measure valued solution $F_t$ $\in$ 
$C([0,\infty);P_2(\bR^3))$ satisfying the property of the moment propagation: 
\begin{equation}\label{propagation-moments}
\mbox{for any $\ell > 0$, $\dis \int|v|^\ell dF_{0} < \infty$ implies $\dis \int|v|^\ell  dF_{t}  \in L^\infty_{loc} ([0,\infty))$}.
\end{equation}

\noindent
(2)When  $-2 \le \gamma <0$,  let 
\begin{align*}
c_{\gamma,s} =\left\{ \begin{array}{lcl}
\max\{\frac{\gamma}{2s} + 1, 0\}&\mbox{if} & 0<s< 1/2,\\
\max\{\gamma+2s, 0\}& \mbox{if} &\gamma +2s < 1 \enskip \mbox{and} \enskip  1/2 \le s <1,\\
\frac{\gamma}{2s-1}+2 &\mbox{if} &  \gamma +2s \ge 1\,\, (, \,\mbox{which implies } \, 1/2 < s).
\end{array}
\right.
\end{align*}
Then,
for any initial datum  $F_0\in P_\a(\bR^3)$ with $\a \in ( c_{\gamma,s}, 2)$,   we have  a measure valued solution $F_t$ $\in$ 
$C([0,\infty);P_{\a}(\bR^3))$ 
to the Cauchy problem \eqref{BE}-\eqref{initial}.
\end{theo}
\begin{rema}\label{remark-main-thm}
Different from the Maxwellian molecule case, $\gamma =0$ (see \cite{toscani-villani}, \cite{morimoto-12}), it is 
unknown that the uniqueness of the (measure valued) solution to the Cauchy problem \eqref{BE}-\eqref{initial}
holds (see \cite{lu-wennberg} about counter examples without the energy conservation law) . 
Instead of \eqref{cross-section2}, we may assume a slightly general condition
\begin{equation}\label{integrability-b}
\mbox{$\exists\varepsilon >0$ such that 
$\dis \int_0^{\pi/2}b(\cos\theta)\sin^{3-2\varepsilon}\theta d\theta<\infty, $}
\end{equation}
which permits the angular cutoff case. 
In the angular cutoff case of hard potentials,  our measure valued solution is the strong solution 
in the sense of
\[
F_t \in C([0,\infty);P_2(\bR^3)) \cap C^1([0,\infty); P_0(\bR^3)),
\]
because it  follows from similar arguments as in \cite{lu-mouhot}, section 5.2, proof of part (a) that 
\[
t \mapsto Q^{\pm}(F_t, F_t) \in C([0,\infty);P_0(\bR^3)),
\]
where we write $Q = Q^+ - Q^-$.
\end{rema}

{ 
Since the weak solution given in Theorem \ref{main-theo} is energy conservative, 
it enjoys the smoothing effect if the initial datum is not a single Dirac mass,
by using 
the support expanding property of the energy conservative solution,
studied by Fournier \cite{Fournier-2015} in the non-cutoff case and 
Pulvirenti-Wennberg \cite{PW0, PW} in the cutoff case.

}

\begin{theo}\label{smoothing-1st-try} 
Let $B(\cdot)$ satisfy \eqref{cross-section}-\eqref{cross-section2} and assume  $2 \ge \gamma
>\max\{-2s, -1\}$. 

\noindent
(1) {\bf Case $\gamma>0$\,:} If 
the initial datum $F_0 \in P_2(\RR^3)$ is not a single Dirac mass, 
then for any $t_0 >0$ there exists a solution $F_t \in C([0,\infty);P_2(\bR^3))$
to the Cauchy problem \eqref{BE}-\eqref{initial}
such that 
$
F_t \in C^\infty([t_0, \infty); \mathcal{S}(\RR^3)).
$

\noindent
(2) {\bf Case $\gamma <0$\,:}
If  $F_0 \in P_2(\RR^3)$ is not a single Dirac mass and its support is compact 
 (or  $\int |v|^\ell dF_0 < \infty$ for any $\ell >0$), then we have a solution 
$
F_t \in C([0,\infty);P_2(\bR^3)) \cap C^\infty((0, \infty); \mathcal{S}(\RR^3))$. 
\end{theo}


As for the existence of weak solutions for the measure valued initial datum, 
 the case $0 < \gamma \le 2$ and $0 < s<1$ was studied in Lu-Mouhot \cite{lu-mouhot}
by approximating the initial measure datum to the one with the density function which satisfies
the finite entropy and the finite moments of any order, through the Mehler transform,
see also Zhang-Zhang \cite{zhang-zhang} for the case $0 < \gamma \le 1$, by the probabilistic method.
The gain of the regularity up to the Besov space $B^s_{1, \infty}$ for a certain $s >0$ was proved in Fournier \cite{Fournier-2015}, 
by the probabilistic method, assuming  the existence of solutions with 
finite moments of any order in hard potential case and a certain order in soft case, if $-1 < \gamma <1$, $0< s <1/2$,
and this regularity gives the  finiteness of the entropy for any positive time.
%
%
For the other reference on the spatially homogeneous Boltzmann equation before 2002, one can refer the the classical review paper by Villani 
\cite{villani2}. 

The plan of this paper is as follows: In Section \ref{S2}, we construct the solution under the cutoff assumption
on both kinetic and angular parts of the cross section, by using the Toscani metric (\cite{Carlen-Gabetta-Toscani, Gabetta-Toscani-Wennberg, toscani-villani, Cannone-Karch-1}). The existence of the solution under the non cutoff assumption is proved in Section \ref{S3},
by using the equi-continuity of the Fourier transform of approximation solutions obtained in Section \ref{S2}, together with
the gain of moments in the hard potential case proved in Section \ref{S4}.
Finally, we show the smoothness of measure valued solutions, applying methods developed in \cite{amuxy7,HMUY, MUXY-DCDS}, 
and using the expanding of the support to the whole space proved in \cite{Fournier-2015}.


\section{Existence under cut-off assumption}\label{S2}
\setcounter{equation}{0}

Following Jacob \cite{Jacob} and Cannone-Karch \cite{Cannone-Karch-1}, call the Fourier transform of a probability measure $F \in P_0(\RR^3)$, that is,
\[
\varphi(\xi) = \hat f(\xi) = \cF(F)(\xi) =\int_{\RR^3} e^{-iv\cdot \xi} dF(v),
\]
a characteristic function.
Put
$\cK = \cF(P_0(\RR^3))$.
Inspired by a series of works by Toscani and his co-authors \cite{Carlen-Gabetta-Toscani, Gabetta-Toscani-Wennberg, toscani-villani}, Cannone-Karch defined a subspace
$\cK^\a$ for $\alpha\ge 0$ as follows:
\begin{align}\label{K-al}
\cK^\alpha =\{ \varphi \in \cK\,;\, \|\varphi - 1\|_{\alpha} < \infty\}\,,
\end{align}
where
\begin{align}\label{dis-norm}
\|\varphi - 1\|_{\alpha} = \sup_{\xi \in \RR^3} \frac{|\varphi(\xi) -1|}{|\xi|^\alpha}.
\end{align}
The space $\cK^\alpha $ endowed with the distance
\begin{align}\label{distance}
\|\varphi - \tilde \varphi\|_{\alpha} = \sup_{\xi \in \RR^3} \frac{|\varphi(\xi) -
	\tilde \varphi(\xi) |}{|\xi|^\alpha}
\end{align}
is a complete metric space (see Proposition 3.10 of \cite{Cannone-Karch-1}).
It follows that $\cK^\alpha =\{1\}$ for all $\alpha >2$ and the
following embeddings
(Lemma 3.12 of \cite{Cannone-Karch-1}) hold
\[
\{1\} \subset \cK^\alpha \subset \cK^\beta \subset \cK^0 =\cK \enskip \enskip
\mbox{for all $2\ge \alpha \ge \beta \ge 0$}.
\]

With this classification on characteristic functions, the global
existence of measure valued solution for the Maxwellian molecules case in $\cK^\alpha$ was studied in \cite{Cannone-Karch-1}(see also \cite{morimoto-12}). It was proved in \cite{MY} that the smoothing effect of 
measure valued solutions always occurs except for a single Dirac mass initial datum. 
After that, the results about the existence and smoothing effect of solutions have been improved 
in \cite{MWY-2014,MWY-2015} by introducing more precise characteristic function spaces $\cM^\a$ and $\widetilde{\cM}^\a$,
which satisfy relations
\begin{align*}
&\mbox{$\cK^\a \supsetneq  \cF(P_\a) =\cM^\a =\widetilde{\cM}^\a  \supsetneq \cK^{\a'}$} \mbox{
 if $0< \a  < \a'  \le 2, \a \ne 1$},\\
&\mbox{$\cK^1 \supsetneq \cM^1 \supsetneq \cF(P_1) =\widetilde{\cM}^1  \supsetneq \cK^{\a'}$}
\,\,\, \mbox{
 if $1 < \a'  \le 2$}.
\end{align*}
For the later use, we recall the space 
$\wt\cM^\alpha$:
\begin{align}\label{tilde-M-al}
\wt\cM^\alpha =\{ \varphi \in \cK\,;\, \|\re\varphi - 1\|_{\cM^\alpha}+||\vp-1||_\a < \infty\}\,,\enskip \a \in (0,2)\,,
\end{align} 
where $\re\vp$ stands for the real part of $\vp(\xi)$ and 
\begin{align}\label{integral-norm}
\|\varphi -1\|_{\cM^\a}  = \int_{\RR^3} \frac{|\varphi(\xi)-1 |}{|\xi|^{3+\a} }d\xi\,.
\end{align}
It should be noted (see \cite{MWY-2015}, ex.) that $\cK^2 = \cF(P_2(\RR^3))$, which is a key of our construction
of the measure valued solution for the finite energy initial datum.

In this part, we consider the Cauchy problem for the cut-off equation, that is,
\begin{align}\label{cutoff-BE}
f_t=\int_{\bR^3}\int_{\bS^2} b(\cdot)\Phi_c(|v-v_*|)(f(v_*')f(v')-f(v_*)f(v))dv_*d\s,
\end{align}
with the initial data $F_0\in P_\beta(\bR^3)$, $0<\b \le 2$, where we assume that $b$ is integrable and without loss of generality, we can set
$$\int_{\bS^2} b(\cdot)d\s=1 . $$
For the hard potential case, $\Phi_c$ takes the form
 $\Phi_c(|z|)=|z|^\g\phi_c(|z|)$, where $\phi_c\in C_0^\infty(\bR^3)$ supporting on $\{z:|z|\le2\}$ satisfies $|\phi_c|\le1$ and $\phi_c=1$ on $\{z:|z|\le1\}$.
While for the soft potentials case, we introduce a different cut-off way of $|v-v_*|^\g$ $(-2<\g<0)$: for any $r\ge2$, let $\psi_r(|z|)\in \cD(\bR)$ be a smooth function satisfying: (1) $\psi_r$ support on $\{z:1/2r\le|z|\le2r\}$; (2) $0\le\psi_r\le1$; (3) $\psi_r=1$ on $\{z:1/r\le|z|\le r\}$. Then, let
 $$
 \Phi_c(|z|)=|z|^\g\psi_r(|z|).
 $$

To investigate the equation \eqref{cutoff-BE}, we start from the Bobylev formula (see \cite{ADVW} and
 \cite{Bobylev-DANS-1975, Bobylev-1988} ) which reads as,
\begin{align}\label{cutoff-BE-bobylev}
\vp_t(\xi)=\int_{\bS^2} b\left(\frac{\xi\cdot\s}{|\xi|}\right)\int_{\bR^3}\hat\Phi_c(\zeta)(\vp(\xi^-+\zeta)\vp(\xi^+-\zeta)-\vp(\zeta)\vp(\xi-\zeta))d\zeta d\s.
\end{align}
where $\hat \Phi_c=\cF(\Phi_c)$.
As usual, we have the following estimates for $\hat \Phi_c$.
\begin{lemm}
	Let $\Phi_c$ be defined as above, then if $\Phi_c(|z|)=|z|^\g\phi_c(|z|)$ with $\g>-3$, we have
	$$
	\forall k\ge0,k\in \bN,\quad |\partial_\zeta^k\hat\Phi_c(\zeta)|\lesssim\frac1{{\la\zeta\ra}^{3+\g+k}};
	$$
	if $\Phi_c(|z|)=|z|^\g\psi_r(|z|)$, for all $N\ge1,k\ge0,N,k\in\bN$ and all $\zeta\in\bR^3,$ there exists $C_{r,N,k}>0$ such that
	$$
	 \quad |\partial_\zeta^k\hat \Phi_c(\zeta)|\lesssim \frac{C_{r,N,k}}{\la\zeta\ra^{2N}}.
	$$
\end{lemm}
\begin{proof}
The proof of the first part of the lemma can be found in \cite{AMUXY-JFA}. When $k=0$, the second estimate follows
\begin{align*}
\left|(1+|\zeta|^2)^N\hat \Phi_c(\zeta)\right|
&=\left|\int[(1-\D_z)^N e^{-i\zeta\cdot z}] \Phi_c (|z|) dz\right|\\
&=\left|\int e^{-i\zeta\cdot z}(1-\D_z)^N \Phi_c (|z|) dz\right|\\
&\le\int_{\{1/2r\le|z|\le1/r\}\cup\{r\le|z|\le 2r\}}\left|(1-\D_z)^N \Phi_c (|z|)  \right|dz+||\Phi_c||_{L^1(\bR^3)}\\
&\le C_{r,N,0}.
\end{align*}
The case $k\neq0$ is similar because $\Phi_c$ supports on a compact subset of $\bR^3$ which doesn't contain the origin.
\end{proof}
We reformulate the problem as
\begin{align}
\vp_t+A\vp=\cG_1(\vp)+\cG_2(\vp),
\end{align}
where $A=\sup_{z\in\bR^3}|\Phi_c(|z|)|$ and
\begin{align}
\cG_1(\vp)(\xi)&=\int_{\bS^2}b\left(\frac{\xi\cdot\s}{|\xi|}\right)\int_{\bR^3}\hat\Phi_c(\zeta)\vp(\xi^-+\zeta)\vp(\xi^+-\zeta)d\s d\zeta;\\
\cG_2(\vp)(\xi)&=A\vp(\xi)-\int_{\bS^2}b\left(\frac{\xi\cdot\s}{|\xi|}\right)\int_{\bR^3}\hat\Phi_c(\zeta)\vp(\zeta)\vp(\xi-\zeta)d\s d\zeta.
\end{align}
Then we can write the equation in the following integral form,
\begin{align}
\vp(t,\xi)=e^{-At}\vp_0+\int_0^t e^{-A(t-\tau)}[\cG_1(\vp(\tau,\xi))+\cG_2(\vp(\tau,\xi))]d\tau.
\end{align}

Our first claim is
\begin{lemm}\label{g-positive-definite}
	For every $\vp\in\cK$, $\cG_1(\vp)$ and $\cG_2(\vp)$ are continuous and positive definite. Moreover,
 if we assume 
\begin{align}\label{alpha-cond}
\left\{ \begin{array}{ll}
\mbox{$0< \alpha < \min\{\gamma, 1\}$ }& \mbox{when $\gamma >0$,}\\ 
\mbox{$0 < \alpha \le 2$} & \mbox{when $\gamma <0$,} 
\end{array} \right.
\end{align}
then 
for any $\vp,\p\in\cK^\a $, there exists $C>0$ such that
	\begin{align}\label{g-estimate}
	|\cG_1(\vp)+\cG_2(\vp)-\cG_1(\p)-\cG_2(\p)|\le(A+C)||\vp-\p||_\a|\xi|^\a,
	\end{align}
	holds for all $\xi\in\bR^3$.
\end{lemm}

\begin{proof}
	To prove $\cG_1(\vp)$ is a characteristic function, let
	$$
	\cG_1^m(\xi)=\int_{\bS^2} b\left(\frac{\xi\cdot\s}{|\xi|}\right) \int_{\bR^3} \hat\Phi_c(\zeta)\vp(\xi^-+\zeta)e^{-\frac{|\xi^-+\zeta|^2}{2m}}\vp(\xi^+-\zeta)e^{-\frac{|\xi^+-\zeta|^2}{2m}}d\zeta d\s.
	$$
	Since $\hat\Phi_c, b$ are integrable, we can apply the dominant convergence theorem to find that
	\[
	\cG_1^m(\xi)\to \cG_1(\vp)(\xi) \mbox{ pointwisely as } m\to \infty.
	\]
	So we only need to prove for each $m$, $\cG_1^m$ is characteristic function. The following proof follows Lemma 2.1 of \cite{PT}.\\
	
	Denote $F=\cF^{-1}(\vp)$, then $\cG_1^m(\xi)$ is the Fourier transform of the positive density
	$$
	G_1^m(F)(v)=\int_{\bR^3}\int_{\bS^2} b\left(\frac{(v-v_*)\cdot \s}{|v-v_*|}\right)\Phi_c(|v-v_*|)f_m(v_*')f_m(v')dv_*d\s
	$$
	where
	\begin{align*}
	f_m(v)=\int_{\bR^3} \o_m(v-u)dF(u),\quad
	\o_m(v)=\left(\frac{m}{2\pi}\right)^{3/2} e^{-\frac{m}{2}|v|^2}.
	\end{align*}
	Similarly, to prove $\cG_2(\vp)$ is a characteristic function, let
	$$
	\cG_2^m(\xi)=A\vp(\xi)e^{-\frac{|\xi|^2}{2m}}-\int_{\bS^2} b\left(\frac{\xi\cdot\s}{|\xi|}\right) \int_{\bR^3} \hat\Phi_c(\zeta)\vp(\zeta)e^{-\frac{|\zeta|^2}{2m}}\vp(\xi-\zeta)e^{-\frac{|\xi-\zeta|^2}{2m}}d\zeta d\s.
	$$
	We can find that
	$$
	\cG_2^m(\xi)\to \cG_2(\vp)(\xi) \mbox{ pointwisely as } m\to \infty.
	$$
	So we only need to prove for each $m$, $\cG_2^m$ is characteristic function. In fact, $\cG_2^m(\xi)$ is the Fourier transform of the positive density
	$$
	G_2^m(F)(v)=\left(A-\int_{\bR^3}\int_{\bS^2} b\left(\frac{(v-v_*)\cdot \s}{|v-v_*|}\right)\Phi_c(|v-v_*|)f_m(v_*)dv_*d\s \right) f_m(v)
	$$
	where
	\begin{align*}
	f_m(v)=\int_{\bR^3} \o_m(v-u)dF(u),\quad
	\o_m(v)=\left(\frac{m}{2\pi}\right)^{3/2} e^{-\frac{m}{2}|v|^2}.
	\end{align*}
	
	Next, to show \eqref{g-estimate}, we first consider the case $\Phi_c(|z|)=|z|^\g\phi_c(|z|)$, that is, hard potential case $\gamma >0$.
We only need to prove
	\begin{align}\label{g-estimate'}
	\int_{\bS^2}b\left(\frac{\xi\cdot\s}{|\xi|}\right)\int_{\bR^3}|\hat\Phi_c(\zeta-\xi^-)-\hat\Phi_c(\zeta)|\frac{|\zeta|^\a+|\xi-\zeta|^\a}{|\xi|^\a} d\zeta d\s\le C
	\end{align}
	for some constant $C>0$. 
	
	For $|\xi|\le1$,
	\begin{align*}
	\int_{\bR^3}&|\hat\Phi_c(\zeta-\xi^-)-\hat\Phi_c(\zeta)|\frac{|\zeta|^\a+|\xi-\zeta|^\a}{|\xi|^\a} d\zeta\\
	&\le\int_{\bR^3}\int_0^1|\pa_\zeta^1\hat\Phi_c(\zeta-\tau\xi^-)|d\tau|\xi^-|\frac{|\zeta|^\a+|\xi-\zeta|^\a}{|\xi|^\a} d\zeta\\
	&\lesssim\sin\frac{\theta}{2}\int_0^1\int_{\bR^3}\frac{|\zeta|^\a+|\xi-\zeta|^\a}{\la\zeta-\tau\xi^-\ra^{3+\g+1}}d\zeta d\tau\\
	&\lesssim\sin\frac{\theta}{2}\int_0^1\int_{\bR^3}\frac{2|\zeta-\tau\xi^-|^\a+1}{\la\zeta-\tau\xi^-\ra^{3+\g+1}}d\zeta d\tau\lesssim C \sin\frac\theta2.
	\end{align*}	
	For $|\xi|>1$,
	\begin{align*}
	\int_{\bR^3}&|\hat\Phi_c(\zeta-\xi^-)-\hat\Phi_c(\zeta)|\frac{|\zeta|^\a+|\xi-\zeta|^\a}{|\xi|^\a} d\zeta\\
	&\le\int_{\bR^3}|\hat\Phi_c(\zeta-\xi^-)|\frac{|\zeta|^\a+|\xi-\zeta|^\a}{|\xi|^\a} d\zeta+\int_{\bR^3}|\hat\Phi_c(\zeta)|\frac{|\zeta|^\a+|\xi-\zeta|^\a}{|\xi|^\a} d\zeta\\
	&\lesssim\int_{\bR^3}\frac{2|\zeta-\xi^-|^\a+2}{\la\zeta-\xi^-\ra^{3+\g}} d\zeta+\int_{\bR^3}\frac{2|\zeta|^\a+1}{\la\zeta\ra^{3+\g}} d\zeta\\
	&\le C.
	\end{align*}
	Therefore, \eqref{g-estimate'} is obtained in the case $\gamma >0$.

For the proof of \eqref{g-estimate} in the case  $\Phi_c  = |z|^\gamma \psi_r(|z|)$, that is, 
$\gamma <0$, 
notice that $\hat \Phi_c(\zeta)$ is real-valued function and satisfies $ \hat \Phi_c(\zeta)=\hat \Phi_c(-\zeta)$.  Then 
we have
\begin{align*}
&\int_{\bS^2} b\Big(\frac{\xi\cdot\s}{|\xi|}\Big)\int_{\bR^3}[\hat\Phi_c(\zeta-\xi^-)-\hat\Phi_c(\zeta)]\vp(\zeta)\vp(\xi-\zeta)d\s d\zeta \\
&=\frac12\int_{\bS^2}b\Big(\frac{\xi\cdot\s}{|\xi|}\Big)\int_{\bR^3}[\hat\Phi_c(-\zeta+\xi-\xi^-)+\hat\Phi_c(-\zeta+\xi^-)\\
& \qquad \qquad \qquad -\hat\Phi_c(-\zeta+\xi)-\hat\Phi_c(-\zeta)]\vp(\zeta)\vp(\xi-\zeta)d\s d\zeta.
\end{align*}
Instead of \eqref{g-estimate'}, we show 
\begin{align}\label{567}
&\int_{\bR^3}\left|\hat\Phi_c(-\zeta+\xi-\xi^-)+\hat\Phi_c(-\zeta+\xi^-) -\hat\Phi_c(-\zeta+\xi)-\hat\Phi_c(-\zeta)\right|\\
& \notag \qquad \qquad \qquad \qquad \qquad \qquad \quad \qquad \qquad \times \frac{|\zeta|^\a+|\xi-\zeta|^\a}{|\xi|^\a}d\zeta <C.
\end{align}
We only consider  the  case $|\xi|\le 1$( since  another case $|\xi|>1$ is easy ).
Using
\begin{align*}
&\Big|\hat\Phi_c(-\zeta+\xi^-)-\hat\Phi_c(-\zeta)-\nabla\hat \Phi_c(-\zeta)\cdot\xi^-\Big|\\
&=\Big|\int_0^1(1-\tau)(\xi^-)^T\nabla^2\hat \Phi_c(-\zeta+\tau\xi^-)(\xi^-)d\tau\Big|
\lesssim|\xi|^2\int_0^1\frac1{\la-\zeta+\tau \xi^-\ra^{2N}}d\tau,
\end{align*}
and
\begin{align*}
&\Big|\hat\Phi_c(-\zeta+\xi^+)-\hat\Phi_c(-\zeta+\xi)-\nabla\hat \Phi_c(-\zeta+\xi)\cdot(-\xi^-)\Big|\\
&=\Big|\int_0^1(1-\tau)(-\xi^-)^T\nabla^2\hat \Phi_c(-\zeta+\xi -\tau\xi^-)(-\xi^-)d\tau\Big|\\
&\lesssim|\xi|^2\int_0^1\frac1{\la-\zeta+\xi-\tau \xi^-\ra^{2N}}d\tau,
\end{align*}
we have
\begin{align*}
&\Big|\hat\Phi_c(-\zeta+\xi-\xi^-)+\hat\Phi_c(-\zeta+\xi^-)-\hat\Phi_c(-\zeta+\xi)-\hat\Phi_c(-\zeta)\Big|\\
&\lesssim\Big|\nabla\hat \Phi_c(-\zeta)\cdot\xi^-+\nabla\hat \Phi_c(-\zeta+\xi)\cdot(-\xi^-)\Big|\\
&\quad+|\xi|^2\int_0^1\frac1{\la-\zeta+ \xi -\tau \xi^-\ra^{2N}}+\frac1{\la-\zeta+\tau \xi^-\ra^{2N}}d\tau.
\end{align*}
Due to 
\begin{align*}
\Big|\nabla\hat \Phi_c(-\zeta)\cdot\xi^--\nabla\hat \Phi_c(-\zeta+\xi)\cdot\xi^-\Big|
&=\Big|
\int_0^1(-\xi)^T\nabla^2\hat \Phi_c(-\zeta+\tau\xi)(\xi^-)d\tau\Big|\\
&\lesssim|\xi|^2\int_0^1\frac1{\la-\zeta+\tau\xi\ra^{2N}}d\tau,
\end{align*}
we obtain
\begin{align*}
&\Big|\hat\Phi_c(-\zeta+\xi-\xi^-)+\hat\Phi_c(-\zeta+\xi^-)-\hat\Phi_c(-\zeta+\xi)-\hat\Phi_c(-\zeta)\Big|\\
&\lesssim|\xi|^2\int_0^1\frac1{\la-\zeta+\tau\xi\ra^{2N}}+\frac1{\la-\zeta+\xi -\tau \xi^-\ra^{2N}}+\frac1{\la-\zeta+\tau \xi^-\ra^{2N}}d\tau.
\end{align*}
Hence, \eqref{567} is clear for $|\xi| \le 1$.

\end{proof}

The existence and uniqueness of the solution to the cut-off equation can be obtained by applying the Banach contraction principle to the operator
\begin{align*}
\cF(\vp)(\xi,t)\equiv e^{-At}\vp_0+\int_0^t e^{-A(t-\tau)}[\cG_1(\vp(\tau,\xi))+\cG_2(\vp(\tau,\xi))]d\tau.
\end{align*}
For $T>0$, which will be determined later, denote $\cX_T^\a=C([0,T],\cK^\a)$, then $\cX_T^\a$ is a complete metric space with respect to the following
distance:
$$
||\vp(\xi,t)-\p(\xi,t)||_{\cX_T^\a}=\sup_{t\in[0,T]}||\vp(\xi,t)-\p(\xi,t)||_\a,\quad \mbox{ for all }\vp,\p\in\cX_T^\a.
$$

Then, we will prove
\begin{lemm}\label{contraction-mapping}
	For all $\vp\in\cK $, $\cF(\vp)$ is continuous and positive definite, and satisfies $\cF(\vp)(t,0)=1$. Moreover, for any $\vp,\tilde{\vp}\in\cK^\a$, there exists $T>0$, which is independent of $\vp_0$, such that $\cF:\cX_T^\a\to \cX_T^\a$ is a contraction mapping.
\end{lemm}
\begin{proof}
	The proof of $\cF(\vp)$ is continuous and positive definite follows Lemma \ref{g-positive-definite}. We can also see that $\cF(\vp)(t,0)=1$ by using the fact $\cG_1(\vp)(t,0)+\cG_2(\vp)(t,0)=A$. Since
	\begin{align*}
	\cF(\vp)-1=e^{-At}(\vp_0-1)+\int_0^te^{-A(t-\tau)}(\cG_1(\vp)+\cG_2(\vp)-\cG_1(1)-\cG_2(1))d\tau,
	\end{align*}
	by Lemma \ref{g-positive-definite}, we have
	\begin{align*}
	\sup_{t\in[0,T]}||\cF(\vp)-1||_\a\le||\vp_0-1||_\a+(A+C)T||\vp-\tilde{\vp}||_{\cX_T^\a}<\infty.
	\end{align*}
	Hence, we proved $\cF$ maps $\cX_T^\a$ to itself.\\
	
	Next, for any $\vp,\tilde{\vp}\in \cX_T^\a$, using Lemma \ref{g-estimate} again, we can obtain
	\begin{align*}
	\sup_{t\in[0,T]}||\cF(\vp)-\cF(\tilde{\vp})||_\a\le(A+C)T||\vp-\tilde{\vp}||_{\cX_T^\a}.
	\end{align*}
	This shows $\cF$ is a contraction mapping if $T>0$ is sufficiently small.
\end{proof}

Lemma \ref{contraction-mapping} shows that there exists a unique solution $\vp(\xi,t)\in C([0,T],\cK^\a)$ to the problem by the Banach fixed point theory. Since the choice of $T$ doesn't depend on the initial data, we can apply the above argument to the initial data $\vp(\xi,T)$ and we can obtain the solution in $C([T,2T],\cK^\a)$. Continuing to do this, we can obtain the following theorem.
\begin{theo}\label{existence-cutoff-Bobylev}
	Let $b,\Phi_c$ be with the same setting as the beginning of this section.  Then for any $\vp_0\in \cK^\a$ 
with $\a$ satisfying
\eqref{alpha-cond}, there exists a unique classical solution $\vp(\xi,t)\in C([0,\infty),\cK^\a)$ to the Cauchy problem of equation \eqref{cutoff-BE-bobylev}.
\end{theo}
As a result, we can obtain
\begin{theo}\label{existence-cutoff-BE}
	Let $b,\Phi_c$ be with the same setting as the beginning of this section. Then for any $F_0\in P_\b$, $0<\b \le 2$, there exists a unique solution $F_t$$\in $$C([0,\infty),P_\a)$ with $\a \in (0, \b)$ satisfying \eqref{alpha-cond}  to the Cauchy problem \eqref{cutoff-BE} in the distribution sense, that is , for any $\psi(v)\in C_b^2(\bR^3)$, the following holds,
	\begin{align}\label{distribution-sense}
	\int\psi(v)dF_t=\int\psi(v)dF_0
+&\frac 1 2\int_0^t\iint\int_{\bS^2}b(\cdot)\Phi_c( \psi(v_*') + \psi(v')\\
&-\psi(v_*)-\psi(v))dF_\tau(v_*)dF_\tau(v)d\s d\tau.\notag
	\end{align}
	Moreover, if $F_0\in P_2$, then the solution satisfies
	\begin{align}\label{energy-conservation}
	\int_{\bR^3}|v|^2dF_t=\int_{\bR^3}|v|^2dF_0,
	\end{align}
	for all  $t>0$.
\end{theo}
\begin{proof}
	The existence and uniqueness of the solution follow Theorem \ref{existence-cutoff-Bobylev} once notice the inclusion $\cF^{-1}(\cK^\b)\subset P_\a$.
	
	To show \eqref{energy-conservation}, we first prove $\int|v|^2dF_t<\infty$.	
	Let $W_\d(v)=\la v\ra^2\la\d v\ra^{-2}$. Using $W_\d(v')\le W_\d(v)+W_\d(v_*)$, we can obtain
	\begin{align*}
	\frac{d}{dt}\int W_\d(v) dF_t(v)
	&=\iint b(\cdot)\Phi_c W_\d(v)(dF_t(v_*')dF_t(v')-dF_t(v_*)dF_t(v))\\
	&\le\iint b(\cdot)\Phi_c(W_\d(v)+W_\d(v'))dF_t(v_*)dF_t(v)\\
	&\lesssim\iint b(\cdot)\Phi_c(W_\d(v)+W_\d(v_*))dF_t(v_*)dF_t(v)\\
	&\lesssim\int W_\d(v)dF_t(v),
	\end{align*}
	Therefore,
	$$
	\int W_\d(v)dF_t(v)<C.
	$$
	Letting $\d\to0$, we have $\int|v|^2dF_t<\infty$. Then, \eqref{energy-conservation} holds by using the fact that $|v|^2$ is collision invariant.
\end{proof}

\section{Existence under non-cutoff assumption}\label{S3}
\setcounter{equation}{0}
In this section, we consider the Cauchy problem \eqref{BE}-\eqref{initial} under the assumptions $\gamma \ge -2$ and \eqref{integrability-b}.
 For each $n\ge2$, let $b_n=\min\{b,n\}$. For the hard potential case, let $\Phi_n(|z|)=\Phi(|z|)\phi_c(|z|/n)$; for the soft potential case, let $\Phi_n(|z|)=\Phi(|z|)\psi_n(|z|)$.
\subsection{Finite energy initial datum}\label{ss1}
By Theorem  \ref{existence-cutoff-BE}, for any $F_0\in P_2$, there exists a unique measure valued solution $F_t^n$ to the equation
\begin{align}\label{ncutoff-hard-BE}
f_t=\int_{\bR^3}\int_{\bS^2} b_n(\cdot)\Phi_n(|v-v_*|)(f(v_*')f(v')-f(v_*)f(v))dv_*d\s,
\end{align}
and $F_t^n$ satisfies \eqref{energy-conservation}. On the other hand, the corresponding solution $\vp^n=\cF(dF_t^n)\in C([0,\infty),\cK^\a)$ solves
\begin{align}\label{ncutoff-hard-BE-bobylev}
\frac{\partial\vp}{\partial t}=\int_{\bR^3}\int_{\bS^2} b_n\left(\frac{\xi\cdot\s}{|\xi|}\right)\int_{\bR^3}\hat\Phi_n(\zeta)(\vp(\xi^-+\zeta)\vp(\xi^+-\zeta)-\vp(\zeta)\vp(\xi-\zeta))d\zeta d\s,
\end{align}
with the initial data $\vp_0=\cF(dF_0)$.
\begin{lemm}\label{equi-continuity-sequence}
	The sequence of solutions $\{\vp^n(\xi,t)\}_{n=1}^\infty$ is bounded in $C(\bR^3\times[0,\infty))$ and equi-continuous on any compact subsets of   $\bR^3\times[0,\infty)$.
\end{lemm}
\begin{proof}\enskip
		{\bf Step 1:}  Uniform bound. $|\vp^n|\le1$.

\noindent 
		{\bf Step 2:} Equi-continuity with respect to $\xi$ in a compact set in $\RR^3_{\xi}$. This follows from
		\begin{align}\label{equi-estimate}
		|\vp^n(\xi)-\vp^n(\xi+\eta)|\le||1-\vp^n||_\a\left(2|\xi|^{\a/2}|\eta|^{\a/2}+|\eta|^\a\right) 
		\end{align}
		and
		\begin{align*}
		\sup_t\sup_{\xi\neq0}\frac{|\vp^n-1|}{|\xi|^\a}&=\sup_t||\vp^n(t)-1||_\a\\
		&\le C_\a\sup_t\int|v|^\a dF_t^n\\
		&\le C_\a\sup_t\int(1+|v|^2) dF_t^n\\
		&\le C_\a\int(1+|v|^2) dF_0.
		\end{align*}
		For the proof of \eqref{equi-estimate}, we refer (19) of  \cite{morimoto-12}.
		
\noindent
		{\bf Step 3:} Equi-continuity with respect to $t$. 
First we consider the soft potential case $0 > \gamma \ge -2$. 
        Let $\p \in C_b^2(\bR^3)$ be a test function in the definition of the measure valued solution. Then for any $0\le s<t\le T$, 
there exists a $C_\p>0$ independent of $n$ such that
\begin{align}
        &\left|\int\p dF^n_t(v)-\int \p dF^n_s(v)\right|\label{uniform-equi-time}\\\notag 
        &\le\frac12\int_s^t\Big|\iiint b_n(\p_*'+\p'-\p_*-\p)d\sigma \\
&\notag \qquad \qquad \times |v-v_*|^\g \psi_n(|v-v_*|)dF^n_\tau(v)dF^n_\tau(v_*) \Big |d\tau\nonumber\\
        &\le |t-s|C_\p \max_{s \le \tau \le t} \iint|v-v_*|^{\g+2}dF^n_\tau(v)dF^n_\tau(v_*)\nonumber\\
        &\le |t-s|C_\p \int\la v\ra^2 dF_0    \enskip \mbox{since $\gamma \ge -2$},\nonumber
        \end{align}
where we used the usual manners (cf, p.290-291 in \cite{villani} ) as follows: If $\dis \mathbf{k} = \frac{v-v_*}{|v-v_*|}$, then 
the Taylor expansion of order 2 for  $\p' -\p$,  
\begin{align*}
\p' -\p &= \nabla \p(v)(v'-v) + \int_0^1(1-\tau) \nabla^2\p(v+ \tau(v'-v))d\tau(v'-v)^2\\
&= \frac{|v-v_*|}{2}\nabla \p(v)\{(\sigma\cdot \mathbf{k})\mathbf{k} -\sigma\} + \nabla \p(v)\frac{v-v_*}{2}(1-
\mathbf{k}\cdot \sigma) + O(|v-v_*|^2\theta^2)
\end{align*}
and it for $\p'_* - \p_*$ imply
\begin{align}\label{justification}
&\p_*'+\p'-\p_*-\p - \frac{|v-v_*|}{2}\{\nabla \p(v)- \nabla \p(v_*)\}\{(\sigma\cdot \mathbf{k})\mathbf{k} -\sigma\} \\
& \qquad \qquad \qquad \qquad \qquad \qquad = O(|v-v_*|^2\theta^2) \nonumber 
\end{align}
because  of $\nabla(\psi - \psi_*) = O(|v-v_*|)$,  and we have 
\begin{align*}
\left|\int_{\SS^2} b_n (\p_*'+\p'-\p_*-\p)d\sigma \right|\le C_\p |v-v_*|^2,
\end{align*}
since  $\int_{\SS^2} b_n(\sigma\cdot \mathbf{k}) \{(\sigma\cdot \mathbf{k})\mathbf{k} -\sigma\} d\sigma =0$
(where we used only $\int_0^{\pi/2} b(\cos \theta) \sin^ 3 \theta d\theta < \infty$ instead of \eqref{integrability-b}).
Since we only consider the compact subset of $[0,\infty)\times\bR^3_\xi$, for any fixed $\xi$ which is bounded, we know $e^{-i\xi\cdot v}\in C_b^2(\bR^3_v)$. This ends the proof once we take $\p=e^{-i\xi\cdot v}$ in \eqref{uniform-equi-time}.

The hard potential case requires a subtle device, the Povzner inequality deeply studied in
\cite{mischler-wennberg}, which is briefly reviewed  in the next section.
Anyway,  similar way to the soft potential case, we estimate
\begin{align*}
        &\left|\int\p dF^n_t(v)-\int \p dF^n_s(v)\right|\\\notag 
        &\le\frac12\int_s^t\Big|\iiint b_n(\p_*'+\p'-\p_*-\p)d\sigma \\
&\qquad \qquad \times 
|v-v_*|^\g \phi_c(|v-v_*|/n)dF^n_\tau(v)dF^n_\tau(v_*) \Big |d\tau.
\end{align*}
Since it follows from  the first order expansion that 
\begin{align*}
&\p_*'+\p'-\p_*-\p - \frac{|v-v_*|}{2}\{\nabla \p(v)- \nabla \p(v_*)\}\{(\sigma\cdot \mathbf{k})\mathbf{k} -\sigma\}= O(|v-v_*|\theta)\\
&\mbox{ hence, }
\quad = O((|v-v_*|\theta)^{2-2\varepsilon}) \enskip \mbox{
for any $0 < \varepsilon <\frac12$},
\end{align*}
 we have 
\begin{align*}
        &\left|\int\p dF^n_t(v)-\int \p dF^n_s(v)\right|\\\notag 
&\le 
C_\psi \int_s^t  \iint \Big(\int_0^{\pi/2} b(\cos \theta) \theta^{3-2\varepsilon} d\theta\Big)\\
&\quad \times |v-v_*|^{\g+2-2\varepsilon} \phi_c(|v-v_*|/n)    dF^n_\tau(v)dF^n_\tau(v_*) d\tau\\
& \le C_\p 
\Big \{ |t-s| \max_{s \le  \tau \le t}  \int \la v \ra^2 dF_\tau^n(v) \int \la v_* \ra^{2}
dF_\tau^n(v_*)\\
&  \quad + \int_s^t \Big(\iint_{\bR^3_{v} \times \bR^3_{v_*} }
{\bf 1}_{\{|v| \ge 2|v_*|\} \cup \{|v_*| \ge 2|v|\}} \\
& \qquad \times |v-v_*|^{2-2\varepsilon}
\min\{|v-v_*|^\gamma, n^\gamma\}  dF_\tau^n(v) dF_\tau^n(v_*)\Big) d \tau\Big\}\\
& \le 
 C_\p \Big\{ |t-s| \Big(\int \la v \ra^2 dF_0(v)\Big)^2 + \int_s^t \Big(
\int  \la v \ra^{2-2\varepsilon}
\min\{\la v \ra^\gamma, n^\gamma\}  dF_\tau^n(v) \Big)d\tau \Big\}.
\end{align*}
Now we use \eqref{limit-hard-strong-bis-n} in the next section, that is, for any $T>0$ there exists a $C_T>0$
independent of $n$ such that 
\[
\int_0^T \Big(
\int  \la v \ra^{2-\varepsilon}
\min\{\la v \ra^\gamma, n^\gamma\}  dF_\tau^n(v) \Big)d\tau < C_T\,.
\]
If $0 \le s < t \le T$ and if $R > 1$, then we have 
\[
\int_s^t \Big(
\int_{\la v \ra \ge R}  \la v \ra^{2-2\varepsilon}
\min\{\la v \ra^\gamma, n^\gamma\}  dF_\tau^n(v) \Big)d\tau \le \frac{C_T}{R^\varepsilon},
\]
which leads us to the equi-continuity in $t$ because
\[
\int_s^t \Big(
\int_{\la v \ra <R }  \la v \ra^{2-2\varepsilon} 
\min\{\la v \ra^\gamma, n^\gamma\}  dF_\tau^n(v) \Big)d\tau \lesssim R^4 |t-s|.
\]
\end{proof}

By Lemma \ref{equi-continuity-sequence}, the Ascoli-Arzel\`{a} theorem, and the Cantor diagonal argument, we can conclude that there exists a subsequence of solutions $\{\vp^{n_k}\}_{n_k}\subset\{\vp^n\}_n$, denoted still by $\{\vp^n\}_{n=1}^\infty$, which converges uniformly in every compact subset of $\bR^3\times[0,\infty)$. Let
$$
\vp(\xi,t)=\lim_{n\to\infty}\vp^n(\xi,t),
$$
then $\vp(\cdot,t)$ is a characteristic function for every $t\ge0$ as the point-wise limit of characteristic functions. Denote $F_t=\cF^{-1}(\vp(\xi,t))$, then $F^n_t\to F_t$ in $\cS'(\bR^3)$. We will prove $F_t$ is a measure valued solution to the equation \eqref{BE}. Before this, we first prove the following lemma:
\begin{lemm}\label{weak-convergence-lemma}
	Let $F^n_t,F_t$ be defined as above.
	
		{\rm (i)} For any $t>0$, we have
		\begin{align}\label{energy-nonincreasing}
		\int|v|^2dF_t\le\int|v|^2dF_0.
		\end{align}

        {\rm (ii)} For any $T>0$ and $\psi\in C(\bR^3)$, satisfying the growth condition $|\psi(v)|\lesssim\la v\ra^l,$ for some $0<l<2$, we have
		\begin{align}\label{weak-convergence}
		\lim_{n\to\infty}\int\psi(v)dF^{n}_t=\int\psi(v)dF_t 
        \quad \mbox{ uniformly to } t\in[0,T].
		\end{align}
		
		{\rm (iii)} For any $T>0$ and $\Psi(v,v_*)\in C(\bR^3\times\bR^3)$,
        satisfying the growth condition
        $|\Psi(v,v_*)|\lesssim\la v\ra^l\la v_*\ra^l,$ for some $0<l<2$, we have
		\begin{align}\label{weak-convergence-product}
        \lim_{n\to\infty}\iint\Psi(v,v_*)dF^n_t(v)dF^n_t(v_*)
        =\iint\Psi(v,v_*)dF_t(v)dF_t(v_*),
		\end{align}
		uniformly to $t\in[0,T]$.
\end{lemm}
\begin{proof}
        (i)We first claim that $\int|v|^2dF_t<\infty$ for any $t>0$. For any $\d\in(0,1),$
		\begin{align*}
		\sup_{\d<|\xi|<\d^{-1}}\frac{|\vp^n-1|}{|\xi|^2}\le||\vp^n-1||_2\le C\int|v|^2dF^n_t\le C\int|v|^2dF_0.
		\end{align*}
		Letting $\d\to0$, we know there exists $C>0$, such that
		\begin{align*}		
        ||\vp(t)-1||_2=\lim_{\d\to0}\sup_{\d<|\xi|<\d^{-1}}\frac{|\vp-1|}{|\xi|^2}
        =\lim_{\d\to0}\lim_{n\to\infty}\sup_{\d<|\xi|<\d^{-1}}
        \frac{|\vp^n-1|}{|\xi|^2}
        <C\int|v|^2dF_0.
		\end{align*}
	Then, since  it follows from $\cK^2 = \cF(P_2(\RR^3))$ (see \cite{MWY-2015}, ex.) that  $\int|v|^2dF_t$ is bounded, 
for any $\e>0$, we can select $R_{\e,t}>0$ large enough such that
		\begin{align*}
		\int_{|v|\ge R_{\e,t}}|v|^2dF_t<\e.
		\end{align*}
		Notice that $F^n_t\to F_t$ in $\cS'(\bR^3)$, we have
		\begin{align*}
		\int|v|^2dF_t\le\lim_{n\to\infty}\int|v|^2\chi_{\{|v|\le R_{\e,t}\}} dF^n_t+\e
		\le\int|v|^2dF_0+\e.
		\end{align*}
		Letting $\e\to0$, we obtain \eqref{energy-nonincreasing}.
		
		(ii)Since for any $R>0$,
		\begin{align*}
		R^{2-l}\int_{|v|\ge R}|\psi(v)|dF^n_t\lesssim \int_{|v|\ge R}\la v\ra^2dF^n_t\le\int\la v\ra^2dF^n_t=\int\la v\ra^2dF_0,
		\end{align*}
		for any $\e>0$, we can find $R_\e>0$ large enough such that
		$$
		\int_{|v|\ge R_\e}|\psi(v)|dF^n_t\le \e.
		$$
		The same argument can be applied to $F_t$. As a result, we can select $R_\e'>0$ large enough, such that
		\begin{align}\label{far-estimate}
		\int_{|v|\ge R_\e'}|\psi(v)|dF^n_t+\int_{|v|\ge R_\e'}|\psi(v)|dF_t
        \le 2\e.
		\end{align}
Let $\chi_{R'_\e}(v)$ be a $ C_0^\infty (\bR^3)$ function such that  $\chi_{R'_\e} =1$ on  $\{ |v| \le R'_\e\}$.
Then, by setting  $\hat \psi = \cF( \chi_{R'_\e}\psi) \in \cS(\RR^3)$ we have
\begin{align}\label{uniform-time-estimate}
\left|\int \chi_{R'_{\epsilon}} \psi dF^n_t -\int \chi_{R'_\e}\psi  dF_t\right|&=
(2\pi)^3 \left|\int \hat{ \psi}(\xi) (\vp^n(\xi,t) -\vp(\xi,t)) d\xi\right|\\
&\lesssim M^{-1} \int_{|\xi| \ge M} \la \xi \ra  \hat{ \psi}(\xi)d\xi + \sup_{|\xi| \le M} |\vp^n(\xi,t) -\vp(\xi,t)|\nonumber
\end{align}
for any arbitrary $M >0$. Finally, \eqref{weak-convergence} follows \eqref{far-estimate} and \eqref{uniform-time-estimate}.
		
		(iii) For any $R>0$, denote
		$$
		K_R=\{(v,v_*):|v|\le R,|v_*|\le R\}
		$$
		which is compact in $\bR^3\times\bR^3$. For any $\e>0$, since
		\begin{align*}
		\iint_{K_R^c}|\Psi(v,v_*)|dF^n_t(v)dF^n_t(v_*)
		&\lesssim\frac1{\la R\ra^{2-l}} \iint_{K_R^c}\la v\ra^2\la v_*\ra^2 dF^n_t(v) dF^n_t(v_*)\\
		&\le\frac1{\la R\ra^{2-l}} \left(\int\la v\ra^2 dF^n_t(v)\right) ^2\\
		&\le\frac1{\la R\ra^{2-l}} \left(\int\la v\ra^2dF_0(v)\right)^2,\\
		\end{align*}
		we can find $R_\e>0$ large enough, such that
		\begin{align*} 
        \iint_{K_{R_\e}^c}|\Psi(v,v_*)|dF^n_t(v)dF^n_t(v_*) +\iint_{K_{R_\e}^c}|\Psi(v,v_*)|dF_t(v)dF_t(v_*)\le\e.
		\end{align*}
		To show \eqref{weak-convergence-product}, it remains to prove
		\begin{align}\label{weak-convergence-compact}
		\lim_{n\to\infty}\iint_{K_{R_\e}}\Psi(v,v_*)dF^n_t(v)dF^n_t(v_*)=\iint_{K_{R_\e}}\Psi(v,v_*)dF_t(v)dF_t(v_*).
		\end{align}
		By the Weierstrass approximation theorem, for any $\e>0$, there exists a polynomial function $\sum_jp_j(v)\tilde p_j(v_*)$ such that
		$$
		\sup_{(v,v_*)\in K_{R_\e}}\Big|\Psi(v,v_*)-\sum_{j=1}^{N(\e)}p_j(v)\tilde p_j(v_*)\Big|<\e.
		$$
		From (ii), using the standard smoothing argument, we can obtain that for each $p_j$ ( resp. $\tilde p_j$), we have
		$$
		\lim_{n\to\infty}\int p_j(v)\chi_{\{|v|\le R_\e\}}dF^n_t(v)
        =\int p_j(v)\chi_{\{|v|\le R_\e\}}dF_t(v),\mbox{ uniformly to }t\in[0,T],
		$$
		that is, for $\e>0$, we can find $N_{j,\e}( \mbox{ resp. } \tilde N_{j,\e})\in\bN$, such that for all $n>N_{j,\e}$, we have
		$$
		\left|\int_{|v|\le R_\e} p_j(v)dF^n_t(v)-\int_{\{|v|\le R_\e\}} p_j(v)dF_t(v)\right|\le\frac{\e}{\a_jN(\e)},
		$$
		where $$\a_j=\sup_{|v|\le R}|p_j(v)|+\sup_{|v_*|\le R}|\tilde p_j(v_*)|.$$
If  $$N=\max_{1\le j\le N(\e)}\{N_{j,\e},\tilde N_{j,\e}\},$$
		then, for all $n>N$, we have
		\begin{align*}
		&\left|\iint_{K_{R_\e}}\Psi(v,v_*)dF^n_t(v)dF^n_t(v_*)-\iint_{K_{R_\e}}\Psi(v,v_*)dF_t(v)dF_t(v_*)\right|\\
		&\le\sum_{j=1}^{N(\e)}\left|\iint_{K_{R_\e}}p_j(v)\tilde p_j(v_*)dF^n_t(v)dF^n_t(v_*)-\iint_{K_{R_\e}}p_j(v)\tilde p_j(v_*)dF_t(v)dF_t(v_*)\right|+2\e\\
		&\le\sum_{j=1}^{N(\e)}\left|\int_{|v_*|\le R_\e}\tilde p_j(v_*)dF^n_t(v_*)\left(\int_{|v|\le R_\e}p_j(v) dF^n_t(v)-\int_{|v|\le R_\e}p_j(v) dF_t(v)\right)\right|\\
		&\quad+\sum_{j=1}^{N(\e)}\left|\int_{|v|\le R_\e} p_j(v)dF_t(v)\left(\int_{|v_*|\le R_\e}\tilde p_j(v_*) dF^n_t(v_*)-\int_{|v_*|\le R_\e}\tilde p_j(v_*) dF_t(v_*)\right)\right|\\
&\quad 
+2\e\\
		&\le 3\e.
		\end{align*}
		Letting $\e\to0$ leads to \eqref{weak-convergence-product}.
\end{proof}
We are ready to  prove the existence part of Theorem \ref{main-theo} in the soft potential case.
\begin{proof}[Proof of (1) of Theorem \ref{main-theo} in the case $-2 \le \gamma <0$]
As in the proof of Lemma \ref{equi-continuity-sequence}, note that
\[
\int_{\SS^2} b(\sigma\cdot \mathbf{k}) \{(\sigma\cdot \mathbf{k})\mathbf{k} -\sigma\} d\sigma = \lim_{n\rightarrow \infty}
 \int_{\SS^2} b_n(\sigma\cdot \mathbf{k}) \{(\sigma\cdot \mathbf{k})\mathbf{k} -\sigma\} d\sigma =0
\]
and $\nabla(\psi - \psi_*) = O(|v-v_*|)$.
Therefore,  if we put
\begin{align}\label{notation-345}
\Delta \psi & := 
\p_*'+\p'-\p_*-\p - \frac{|v-v_*|}{2}\{\nabla \p(v)- \nabla \p(v_*)\}\{(\sigma\cdot \mathbf{k})\mathbf{k}-\sigma\}\\
&( = O(|v-v_*|^2\theta^2)  \enskip ), \notag
\end{align}
then  we have
\begin{align*}
&\iint\int_{\bS^2}b_n|v-v_*|^\g\p_n \left(|v-v_*|\right)(\p_*'+\p'-\p_*-\p)dF_t^n(v)dF_t^n(v_*)d\s\\
&\qquad\qquad\qquad\qquad -\iint \int_{\bS^2}b|v-v_*|^\g (\p_*'+\p'-\p_*-\p)dF_t(v)dF_t(v_*)d\s\\
& \quad = \iint\int_{\bS^2}b_n|v-v_*|^\g\p_n \left(|v-v_*|\right) \Delta \p dF_t^n(v)dF_t^n(v_*)d\s\\
&\qquad\qquad\qquad\qquad -\iint \int_{\bS^2}b|v-v_*|^\g \Delta \p dF_t(v)dF_t(v_*)d\s\\
& \quad = \iint  \Big(
\Psi_n(v,v_*) - \Psi(v,v_*)\Big)
 dF_t^n(v)dF_t^n(v_*)\\
&\qquad        +\Big(\iint    \Psi(v,v_*)   dF_t^n(v)dF_t^n(v_*) -   \iint \Psi(v,v_*)  dF_t(v)dF_t(v_*) \Big)
\end{align*}
where 
$\Psi_n(v,v_*) $ $=   |v-v_*|^\g\p_n \left(|v-v_*|\right)$ $ \int_{\bS^2} b_n \Delta \p d\s$ and 
$\Psi(v,v_*)  =  |v-v_*|^\g$ \par 
\noindent $\int_{\bS^2} b \Delta \p d\s$. Note that $\Psi_n, \Psi$ 
satisfy the growth condition of Lemma \ref{weak-convergence-lemma}, (iii) with $\ell = 2+\gamma <2$.
Since, for any $R>0$, we have  
\begin{align*}
&\left| \iint_{\{|v| > R\}\cup \{|v_*|>R\}}  \Big(
\Psi_n(v,v_*) - \Psi(v,v_*)\Big)
 dF_t^n(v)dF_t^n(v_*)\right| \\
& \lesssim R^{\ell-2} \Big(\int \la v \ra^2 dF_0(v)\Big)^2,
\end{align*}
the first term converges to $0$, as $n \rightarrow \infty$, because $\Psi_n(v,v_*)$ converges 
to $\Psi(v,v_*)$ uniformly on a compact set of $\RR^3_v\times \RR^3_{v_*}$.
By Lemma \ref{weak-convergence-lemma}, (iii), we also see 
the second term tends to $0$, as $n \rightarrow \infty$.  Thus $F_t$ is a measure valued 
solution. 
Following the same method as in p.292-293 of \cite{villani}, we can show \eqref{definition2-1}.
In fact, by considering  $\psi =\psi_m =  |v|^2 \phi_c(|v|/m) \in C_b^2(\RR^3)$, we have
\begin{align*}
&\int \psi_m  dF_t(v) - \int \p_m dF_0(v) \\
&\quad =\int_0^t \left( \iint \int_{\bS^2} b|v-v_*|^\g (\p_{m,*}^{\prime}+ \p_m^{\prime} -\p_{m,*}
-\p_m)dF_\tau(v)dF_\tau(v_*)d\s\right) d\tau.
\end{align*}
Since it follows from \eqref{justification} that the right hand side is uniformly bounded with respect to $m$,
for any $\eta >0$ there exists a $\delta >0$ independent of $m$ such that
\[
\left|\iiint_{\{|v-v_*| < \delta\} \cup \{|\theta| < \delta\}} b|v-v_*|^\g (\p_{m,*}^{\prime}+ \p_m^{\prime} -\p_{m,*}
-\p_m)dF_t(v)dF_t(v_*)d\s\right|< \eta,
\]
and hence,  the Lebesgue convergence theorem shows
\[
\lim_{m \rightarrow \infty} \left|\int \psi_m  dF_t(v) - \int \p_m dF_0(v)\right| \le \eta,
\]
which yields  \eqref{definition2-1}. 	It remains to show 
$F_t \in C([0,\infty); P_2(\RR^3))$.
Taking the limit of \eqref{uniform-equi-time}, we obtain $F_t \in C([0,\infty); P_0(\RR^3))$.
This, together with \eqref{definition2-1}, shows $F_t \in C([0,\infty); P_2(\RR^3))$,
since it follows from \cite[Lemma 1]{toscani-villani} (see also \cite[Theorem 7.12]{villani3}) that for any $t_0 \in [0,\infty)$ 
\begin{align}\label{use-234}
\lim_{R \rightarrow \infty} \Big( \mathop{\mbox{\rm lim sup}}_{t \rightarrow t_0} \int_{|v| \ge R} |v|^2 dF_t(v)\Big) =0\,.
\end{align}
The property of the moment propagation \eqref{propagation-moments} is a direct consequence
of the Povzner inequality mentioned in the next section.
\end{proof}

We now consider the hard potential case, admitting some results in the next section.
\begin{proof}
[Proof of (1) of Theorem \ref{main-theo} when $\gamma >0$]
We consider the convergence of the difference 
\begin{align*}
&\int_0^t \iint\int_{\bS^2}b_n|v-v_*|^\g\p_n \left(|v-v_*|\right)(\p_*'+\p'-\p_*-\p)dF_\tau^n(v)dF_\tau^n(v_*)d\s\\
&\qquad\qquad -\int_0^t \iint \int_{\bS^2}b|v-v_*|^\g (\p_*'+\p'-\p_*-\p)dF_\tau(v)dF_\tau(v_*)d\s d\tau \\
& \quad = \int_0^t \iint \Psi_n(v,v_*)  dF_\tau^n(v)dF_\tau^n(v_*)d\s d\tau \\
&\qquad\qquad\qquad -\int_0^t \iint \Psi(v,v_*) dF_\tau(v)dF_\tau(v_*)d\s d\tau, 
\end{align*}
where $\Psi$ is the same as in the soft potential case and 
\begin{align*}
\Psi_n(v,v_*) &=   |v-v_*|^\g \phi_c \left(\frac{|v-v_*|}{n}\right) \int_{\bS^2} b_n \Delta \p d\s. 
\end{align*}
Note that
\begin{align*}
&|\Psi(v,v_*)| \lesssim  |v-v_*|^{\g +2-2\varepsilon}
\lesssim \la v \ra^{2-\varepsilon}\la v_* \ra^{2-\varepsilon}{\bf 1}_{\{|v| < 2|v_*| < 4|v|\}}\\
&\qquad \qquad \qquad + \la v \ra^{\g +2-2\varepsilon} {\bf 1}_{\{|v| \ge 2 |v_*|\}}+ 
 \la v_* \ra^{\g +2-2\varepsilon }{\bf 1}_{\{|v_*| \ge 2 |v|\}},\\
&|\Psi_n (v,v_*)| 
\lesssim \la v \ra^{2-\varepsilon}\la v_* \ra^{2-\varepsilon}{\bf 1}_{\{|v| < 2|v_*| < 4|v|\}}\\
&  + \la v \ra^{2-2\varepsilon}
\min\{\la v \ra^\gamma, n^\gamma\}  {\bf 1}_{\{|v| \ge 2 |v_*|\}}+ 
 \la v_* \ra^{2-2\varepsilon}
\min\{\la v_* \ra^\gamma, n^\gamma\} {\bf 1}_{\{|v_*| \ge 2 |v|\}}.\\
\end{align*}
Therefore, it follows from \eqref{limit-hard-strong-bis-n} in the next section that, 
if $t \le T$ and $R >1$, then we have
\begin{align*}
&\left|
\int_0^t \iint_{\{|v| >R\}\cup \{|v_*| >R\}} \Psi_n(v,v_*)  dF_\tau^n(v)dF_\tau^n(v_*)d\s d\tau\right|\\
&\lesssim T R^{-2\varepsilon} \Big(\int \la v \ra^2 dF_0(v)\Big)^2 
+ R^{-\varepsilon} \int_0^T  \Big(
\int  \la v \ra^{2-\varepsilon}
\min\{\la v \ra^\gamma, n^\gamma\}  dF_\tau^n(v) \Big)d\tau \\
&\lesssim R^{-\varepsilon}.
\end{align*}
Thanks to \eqref{limit-hard-strong-bis},  we have similarly
\begin{align*}
\left|
\int_0^t \iint_{\{|v| >R\}\cup \{|v_*| >R\}} \Psi(v,v_*)  dF_\tau (v)dF_\tau (v_*)d\s d\tau\right|
\lesssim R^{-\varepsilon}.
\end{align*}
Consequently, it suffices to show the convergence of 
\begin{align*}
&\int_0^t \iint_{\{|v| \le R\}\cap \{|v_*| \le R\}} \Psi_n(v,v_*)  dF_\tau^n(v)dF_\tau^n(v_*)d\tau \\
&\qquad\qquad\qquad -\int_0^t \iint_{\{|v| \le R\}\cap \{|v_*| \le R\}}  \Psi(v,v_*) dF_\tau(v)dF_\tau(v_*)d\tau \\
&= \int_0^t \iint_{\{|v|, |v_*| \le R\}} \big( \Psi_n(v,v_*) -\Psi(v,v_*) \big)  dF_\tau^n(v)dF_\tau^n(v_*) d\tau\\
&\quad + \int_0^t \Big( \iint_{\{|v|, |v_*| \le R\}} \Psi(v,v_*) dF_\tau^n(v)dF_\tau^n(v_*)  \\
&\qquad \qquad \qquad \qquad - \iint_{\{|v|, |v_*|\le R\}} \Psi(v,v_*) dF_\tau(v)dF_\tau(v_*) \Big) d\tau.
\end{align*}
Estimations of the first and the second terms are quite the same as for the soft potentials. 
In the hard potential case it remains to show the energy conservation law and
the moment gain property, which will be sent to the next section. 
\end{proof}

\subsection{Infinite energy initial datum for the soft potential}\label{ss2}
For the proof of (2) of Theorem \ref{main-theo}, it first show the following lemma, which is a variant
of  (i)  of Lemma \ref{weak-convergence-lemma}.

\begin{lemm}\label{soft-imp}  Let $-2 \le \gamma <0$ and let $c_{\gamma,s}  < \a <2$.
If $F^n_t$ is a solution in $C([0,\infty); P_{\a'})$, $0 < \a' <\a$,  to the Cauchy problem \eqref{cutoff-BE}
for the initial datum $F_0 \in P_\a$ ( given by Theorem  \ref{existence-cutoff-BE} ),
then for any $T >0$ there exists a $C_T >0$ independent of $n$ such that

		\begin{align*}
		\sup_{t \in [0,T]} \int \la v\ra^\a dF^n_t\le  C_T\int \la v\ra^\a dF_0.
		\end{align*}

\end{lemm}
\begin{proof}
		Using the weight function $W_\d(v)$, it is easy to show $\int\la v\ra^\a dF_t^n<\infty$ for any fixed $n$.
First consider the case $s \ge 1/2$. If $\psi_\a = \la v \ra^\a$  then we have 
		\begin{align*}
		\frac d{dt}\int \psi_\a dF^n_t(v)&=\frac 1 2 \iiint B_n(\cdot)
\Delta \psi_\a
d\s dF^n_t(v)dF^n_t(v_*) \\
&= \frac 1 2 \iiint_{|v-v_*| <1}  \qquad + \frac 1 2 \iiint_{|v-v_*| \ge 1} \,,
\end{align*}
where $\Delta \psi_\a$ is defined by the same formula as  \eqref{notation-345}, and it satisfies 
$\Delta \psi_\a = O(|v-v_*|^2\theta^2)$ because $|\psi_\a^{\prime \prime} |\lesssim 1$. 
By using the mean value theorem to $\psi_{\a,*}' -\psi_{\a,*}$ and $\psi_\a' - \psi_\a$,
notice that $|\Delta \psi_\a | \lesssim |v-v_*| \theta \big(\la v \ra^{(\a-1)^+}+ \la v_* \ra^{(\a-1)^+}\big)$, 
where $a^+ = \max \{a, 0\}$ for $a \in \RR$.  Thus, we have
\begin{align}\label{hyoka}
|\Delta \psi_\a| \lesssim \theta^{1+\delta}|v-v_*|^{1+\delta}
\Big(\la v \ra^{(\a-1)^+}+ \la v_* \ra^{(\a-1)^+}\Big)^{(1-\delta)}\,,\enskip \forall 
\delta \in (0, 1].
\end{align}
When $\gamma +2s \ge 1$, choose an $\varepsilon >0$ such that
$\a > \frac{\gamma}{2s-1+\varepsilon} +2 >c_{\gamma,s}( >1)$ and use \eqref{hyoka} with $\delta = 2s-1+\varepsilon$ to the integral for $|v-v_*| \ge1$.  Using \eqref{hyoka} with $\delta =1$   also to 
the integral for $|v-v_*| <1$, we see easily that 
there exists a $C >0$ independent of $n$ such that
\begin{align*}
		\frac d{dt}\int \psi_\a dF^n_t(v)	\le C  \int \psi_\a dF^n_t(v),
\end{align*}		
which leads to the desired estimate.  When $\gamma + 2s <1$, choose an $\varepsilon >0$
such that $\gamma +2s -1 < -\varepsilon$ if the range of $\a >1$ is handled.
The range $1 \ge \a > c_{\gamma,s} $ is easier, by choosing an $\varepsilon >0$ such that 
$\a > c_{\gamma,s} + \varepsilon$.

For the case $0<s <1/2$, we put simply $\Delta \psi_\a = \psi_{\a,*}' + \psi_\a' -\psi_{\a,*} - \psi_\a$.
Then $\Delta \psi_\a =O(|v-v_*|^2\theta)$,
$=O\big( |v-v_*| \theta \big(\la v \ra^{(\a-1)^+}+ \la v_* \ra^{(\a-1)^+}\big)\big)$
 and  $ = O(\la v \ra^\a + \la v_*\ra^\a)$. The latter two yield
\[
|\Delta \psi_\a| \lesssim \theta^{\delta}|v-v_*|^{\delta}
\Big(\la v \ra+ \la v_* \ra \Big)^{(1-\delta)\a + \delta (\a-1)^+}\,,\enskip \forall 
\delta \in (0, 1].
\]
from which the corresponding estimation is obvious.
%
%
\end{proof}

By this lemma we see $F_t^n \in P_\a(\RR^3)$, noting additionally the fact  
that $v_j$ is collision invariant,  when $\a >1$. Using \cite[(1.16)]{MWY-2015}
for $\varphi^n(\xi,t) = \cF(F^n_t)$, we have
\begin{align}\label{base-equiv}
	\int\frac{|1-\re\vp^n(\xi, t)|}{|\xi|^{3+\a}}d\xi 
= c_\a \int|v|^\a dF_t^n(v),
	\end{align}
where $0< c_\a = 2\int\frac{\sin^2({\zeta\cdot\s}/2)}{|\zeta|^{3+\a}}d\zeta, \sigma \in \SS^2$.
Since $\{\varphi^n(\xi,t)\}_{n=1}^\infty$ converges to
$\varphi(\xi,t)= \cF(F_t)$ uniformly in every compact subset of $\RR^3\times [0,\infty)$,
it follows from \cite[(1.16)]{MWY-2015} that for any $t \in (0,T]$
\begin{align*}
c_\a \int |v|^\a dF_t(v) &= \lim_{\delta \rightarrow +0}
\int_{\delta \le |\xi| \le \delta^{-1}}      \frac{|1-\re\vp(\xi, t)|}{|\xi|^{3+\a}}d\xi\\
&= \lim_{\delta \rightarrow +0} \lim_{n \rightarrow \infty} 
\int_{\delta \le |\xi| \le \delta^{-1}}      \frac{|1-\re\vp^n(\xi, t)|}{|\xi|^{3+\a}}d\xi\\
& \le  C_T c_\a \int \la v \ra^\a dF_0(v)\,,
\end{align*}
from which we can see Lemma \ref{weak-convergence-lemma}
holds with the index $2$ replaced by $\a$. Thus, as in the proof of (1) of Theorem  \ref{main-theo},
we see that, for any $0< \a' < \a$, $F_t \in C((0, \infty); P_{\a'}(\RR^3))$ is a measure valued solution.
Since $F_t $ belongs to $P_{\a}(\RR^3)$  uniformly with respect to $t$ in any finite interval in  $[0, \infty)$,
we obtain 
\[
\left| \int \la v \ra^\a dF_t(v) - \int \la v \ra^\a dF_s(v) \right| = O(|t-s|)
\]
by means of the same estimation for $\Delta \psi_\a$ as in the proof of Lemma \ref{soft-imp},
which shows $F_t \in C((0, \infty); P_{\a}(\RR^3))$ on account of \cite[Theorem 7.12]{villani3}.
Thus, the proof of (2) of Theorem \ref{main-theo}
is completed.  

\section{Povzner inequality and the moment gain property}\label{S4}

\setcounter{equation}{0}

\subsection
{ Case: Initial datum $F_0$ with finite $2+\kappa$ moment. }\label{5-1}
In order to show the properties of moment gain and the moment propagation, we consider the case
where  $F_0$ satisfies, for $\kappa >0$,
\begin{align}\label{high-moment}
\int \la v \ra^{2+\kappa} dF_0 < \infty.
\end{align}
Let $F_t^n$ denote the  solution to  the cutoff equation \eqref{ncutoff-hard-BE}, for both hard and potential cases.
Considering $W_\delta(v) = \la v \ra^{2+ \kappa}\la \delta  v \ra^{-2-\kappa}$ as in the proof of Theorem 
\ref{existence-cutoff-BE} and letting $\delta \rightarrow 0$, we have
\begin{equation}\label{k-moment}
\int \la v \ra^{2+\kappa} dF^n_t (v) < \infty \enskip \mbox{for any $t>0$}.
\end{equation}
We will show the upper bound is independent of $n$, more precisely, 
for any fixed $T >0$ there exists a $C_T >0$ independent of $n$ such that
\begin{align}\label{uniform-hard-strong}
\sup_{t \in [0,T] } \int_{\RR^3_v} \la v \ra^{2+\kappa} dF_t^n(v) + \int_0^T \int_{\RR^3_v}
\la v \ra^{2+\kappa} \min \{\la v \ra^\gamma, n^\gamma\} dF^n_{\tau}(v) d\tau \le C_T,
\end{align}
by using the Povzner inequality deeply studied in \cite{mischler-wennberg}
(see also, the proof  of Corollary 1.7 of \cite{MWY-2015}). 
Since $\sigma \in \SS^2$, it can be  written as
\begin{align*}
\sigma 
=\mathbf{k}\cos\theta+\sin\theta(\mathbf{h}\cos\varphi+ \mathbf{i}\sin\varphi),
\enskip \theta \in [0,\pi), \varphi \in [-\pi,\pi)\,,
\end{align*}
by an orthogonal basis in $\RR^3$, 
\begin{align*}
\mathbf{k}=\frac{v-v_*}{|v-v_*|},\ 
\mathbf{i}=\frac{v\times v_*}{|v\times v_*|},\ 
\mathbf{h}=\mathbf{i} \times \mathbf{k}= \frac{
\big((v-v_*)\cdot v\big) v_* -\big((v-v_*)\cdot v_*\big) v}{|v-v_*| |v \times v_*|}.
\end{align*}
It follows from $(v+v_*)\perp \mathbf{i}$ and the definition of $\mathbf{h}$ that
\begin{align*}
|v'|^2=&|\frac{v+v_*}{2}|^2+|\frac{v-v_*}{2}|^2+ \frac{|v-v_*|}{2}(v+v_*)\cdot \sigma\\
\quad =&\frac{1}{4}(2|v|^2+2|v_*|^2)+\frac{|v-v_*|}{2}\Big((v+v_*)\cdot
(\cos\theta \mathbf{k}+\sin\theta\cos\varphi \mathbf{h}) \Big) \\
\quad =&\frac{1}{2} (|v|^2+|v_*|^2)+\frac{\cos\theta}{2}(|v|^2-|v_*|^2)\\
&\ +\frac{\sin\theta\cos\varphi}{2|v \times v_*|}\Big \{ (v+v_*)
\cdot \Big(\big((v-v_*)\cdot v\big) v_* -\big((v-v_*)\cdot v_*\big) v \Big)\Big\} \\
=&\frac{|v|^2(1+\cos\theta) }{2}+\frac{|v_*|^2(1-\cos\theta)}{2}+|v||v_*|\sin\alpha\sin\theta\cos\varphi \,,
\end{align*}
where $\alpha$ is the angle between $v$ and $v_*$.  Therefore, we have
\begin{align}\label{strong-case}
|v'|^2 &= |v|^2 \cos^2 \frac{\theta}{2} +|v_*|^2 \sin^2 \frac{\theta}{2} +|v \times v_*|\sin\theta\cos\varphi\\
&:= Y(\theta) + Z(\theta) \cos \varphi\,.\notag
\end{align}
{It should be noted that 
\[
|Z(\theta)| = |(v-v_*)\times v_*| \sin \theta \le |v-v_*||v_*| \sin \theta\,,
\]
which will be useful to absorb the singularity of the kinetic part in the soft potential case.}
Similarly, we have
\begin{align}\label{strong-case-2}
|v'_*|^2 &= |v_*|^2 \cos^2 \frac{\theta}{2} +|v|^2 \sin^2 \frac{\theta}{2} -|v \times v_*|\sin\theta\cos\varphi\\
&= Y(\pi-\theta) - Z(\theta) \cos \varphi\,.\notag
\end{align}
 
If we set  $\Psi(x) = \Psi_\kappa(x) =(1+x)^{1+\kappa/2}$, it follows from   the cutoff equation \eqref{ncutoff-hard-BE} that
Furthermore we will show that 
To this end, we first notice  
\begin{align}\label{equation-345}
&\frac{d }{dt} \int \la v \ra^{2+\kappa}  dF^n_t(v )=  \frac{1}{2}\iint \Phi_n(|v-v_*|)  K_n(v,v_*)
dF_t^n(v)dF_t^n(v_*),
\end{align}
where
\begin{align*}
 &K_n(v,v_*)= \int_{\SS^2} b_n 
\big\{\Psi(|v'|^2) + \Psi(|v'_*|^2) -\Psi(|v|^2) -\Psi(|v_*|^2)\big\}d\sigma\\
\quad &=2 \int_0^\pi \int_0^\pi b_n(\cos \theta) \big\{\Psi(|v'|^2) + \Psi(|v'_*|^2) -\Psi(|v|^2) -\Psi(|v_*|^2)\big\}\sin \theta
d\theta d \varphi\,.
\end{align*}
It should be noted that the right hand side of \eqref{equation-345} is well defined because of \eqref{k-moment}.
{
Note that
\begin{align*}
&\int^{\pi}_0 \Psi (Y(\theta )+ Z(\theta )\cos\varphi ) \ d\varphi\\
&=(\int^{\frac{\pi}{2}}_0 + \int^{\pi}_{\frac{\pi}{2}}) \ \Psi (Y(\theta ) + Z(\theta ) \cos\varphi ) \ d\varphi\\
&=\int^{\frac{\pi}{2}}_0 \{ \Psi (Y(\theta )+ Z(\theta )\cos\varphi )+\Psi (Y(\theta )- Z(\theta )\cos\varphi )-2\Psi (Y(\theta )) \} \ d\varphi \\
&\qquad \qquad \qquad \qquad \qquad \qquad  \qquad \qquad \qquad \qquad \qquad \qquad \qquad \qquad+\pi \Psi (Y(\theta )),
\end{align*}
and by using integration by parts twice, we have
\begin{align}\label{recall}
&\int_0^\pi \Psi(|v'|^2) d \varphi = \int^{\pi}_0 \Psi (Y(\theta )+ Z(\theta )\cos\varphi ) \ d\varphi \\
&\notag=\pi \Psi (Y) +[\varphi \{ \Psi (Y+ Z\cos\varphi )+\Psi ( Y- Z \cos\varphi )-2\Psi (Y) \} ]^{\frac{\pi}{2}}_0\\
&\notag \quad -\int^{\frac{\pi}{2}}_0 \varphi \{ {\Psi}'(Y+Z \cos\varphi ) - {\Psi}'(Y- Z\cos\varphi ) \} (- Z\sin\varphi )\ d\varphi \\
&\notag =\pi \Psi (Y) + \int^{\frac{\pi}{2}}_0  Z \varphi \sin\varphi  ({\Psi}'(Y + Z\cos\varphi )-{\Psi}'(Y- Z \cos\varphi ))\ d\varphi \\
&\notag =\pi \Psi (Y) +Z[(\sin\varphi - \varphi \cos\varphi ) \{ {\Psi}'(Y+ Z \cos\varphi ) - {\Psi}'(Y-Z \cos\varphi )\} ]^{\frac{\pi}{2}}_0\\
&\notag + Z^2\int^{\frac{\pi}{2}}_0 (\sin\varphi - \varphi \cos\varphi ) \{ {\Psi}''(Y+ Z \cos\varphi ) +{\Psi}''(Y- Z \cos\varphi ) \} \sin\varphi \ d\varphi \\
&\notag =\pi \Psi (Y(\theta))   +Z^2\int^{\frac{\pi}{2}}_0 (\sin\varphi - \varphi \cos\varphi )\sin\varphi \\
& \notag \qquad \times  \{ ({\Psi}''(Y(\theta)+ Z\cos\varphi ) +{\Psi}''(Y(\theta)-Z \cos\varphi ) \} \ d\varphi .
\end{align}
Similarly, we obtain
\begin{align*}
&\int_0^\pi \Psi(|v'_*|^2) d \varphi 
=\pi \Psi (Y(\pi-\theta))   +Z^2\int^{\frac{\pi}{2}}_0 (\sin\varphi - \varphi \cos\varphi )\sin\varphi \\
& \qquad \qquad \times  \{ ({\Psi}''(Y(\pi-\theta)+ Z\cos\varphi ) +{\Psi}''(Y(\pi-\theta)-Z \cos\varphi ) \} \ d\varphi .
\end{align*}
In view of these formula,}
we can divide $K_n(v,v_*)$ into two parts
\[
K_n(v,v_*) = -H_n(v,v_*) + G_n(v,v_*),
\]
where  it follows from the convexity of $\Psi$ that 
\begin{align*}
&-H_n(v,v_*) = 2\pi \int_0^\pi b_n(\cos \theta)\sin \theta \\
&\quad \times \Big [\{ \Psi(|v|^2 \cos^2 \frac{\theta}{2} +|v_*|^2 \sin^2 \frac{\theta}{2}) - 
\cos^2 \frac{\theta}{2} \Psi(|v|^2) - \sin^2 \frac{\theta}{2}\Psi(|v_*|^2 )\} \\
&\qquad + \{ \Psi(|v_*|^2 \cos^2 \frac{\theta}{2} +|v|^2 \sin^2 \frac{\theta}{2}) - \cos^2 \frac{\theta}{2} \Psi(|v_*|^2 )- 
 \sin^2 \frac{\theta}{2} \Psi(|v|^2) \} \Big ]d\theta \le 0.
\end{align*}
On the other hand, if $Z_0 = Z(\theta)/(1+Y(\theta)) \in [0,1]$, then
\begin{align*}
&Z^2\int^{\frac{\pi}{2}}_0 (\sin\varphi - \varphi \cos\varphi )\sin \varphi \{ {\Psi}''(Y+ Z \cos\varphi ) +{\Psi}''(Y- Z \cos\varphi ) \} \ d\varphi \\
&\le O(|\kappa|)
Z^2(1+ Y)^{\kappa/2-1} \int_0^{\pi/2} \varphi^3 \{(1 + Z_0 \cos \varphi)^{\kappa/2-1} +  (1 - Z_0 \cos \varphi)^{\kappa/2-1}\}d\varphi\\
&\left\{
\begin{array}{ll}
\lesssim  Z^2 \lesssim |v|^2|v_*|^2 \theta^{2}, \enskip &\mbox{if}\enskip \kappa <2;\\\\
\lesssim Z^2(1+ Y)^{\kappa/2-1} \lesssim  |v|^2|v_*|^2\big ( \la v \ra^{\kappa-2} +  \la v_* \ra^{\kappa-2} \big)
 \theta^2,\enskip &\mbox{if} \enskip \kappa \ge 2\,.
\end{array}
\right .
\end{align*}
Consequently, 
there exist $C_0 >0$ and $C_1$  independent of $n$ such that 
\begin{align*}
G_n(v,v_*)  \left\{ \begin{array}{ll}
\le C_0 |v|^2 |v_*|^2\int_0^\pi b_n(\cos \theta) \sin^3 \theta d\theta\le C_1 |v|^2 |v_*|^2&\mbox{if}\enskip \kappa <2;\\\\
\le C_1  \big ( |v_*|^2 \la v \ra^{\kappa} +  |v|^2\la v_* \ra^{\kappa} \big) &\mbox{if}\enskip \kappa \ge 2. 
\end{array}
\right .
\end{align*}
When $\gamma<0$ we have
\begin{align}\label{imp-567-soft}
&\int_{\RR^3} \la v \ra^{2+\kappa}  dF^n_t(v )  \\
&\le \int_{\RR^3} \la v \ra^{2+\kappa}  dF_0(v ) \notag  
 + C'_3\Big ( \int_{\RR^3} \la v \ra^2 dF_0(v) \Big)\int_0^t \int_{\RR^3_v}
\la v \ra^{\max\{2, \kappa\}}  dF^n_{\tau}(v) d \tau, \notag
\end{align}
which shows that for any $T >0$ there exists a $C_T >0$ independent of $n$ such that
\[
\sup_{0 \le t \le T} \int_{\RR^3} \la v \ra^{2+\kappa}  dF^n_t(v )  \le C_T.
\]
Taking the limit as $n \rightarrow \infty$, we obtain the property of the moment propagation of
the solution $F_t$ for 
soft potentials. .

From  now on we consider the hard potential case.  Note that there exist 
 $[\theta_1, \theta_2] \subset (0, \pi/2)$ and $c_0>0$ independent of $n$ such that
$ b_n(\cos \theta)\sin \theta  \ge c_0 $ on $[\theta_1, \theta_2] $. If $A, B >0$, $A \ne B$ and if  $F(x) = 
\Psi(A) +\Psi(B) -\Psi(Ax + B(1-x)) - \Psi(A(1-x)+Bx)  $, then it follows from the convexity of $\Psi$ that $F(x)$ takes the maximum at $x =1/2$ and is decreasing with respect to
$|x-1/2|$. Therefore we have 
\begin{align*}
H_n(v,v_*) & \ge 2\pi c_0 (\theta_2 - \theta_1)
 \Big [
  \Psi(|v|^2)  + \Psi(|v_*|^2 ) 
- \Psi(|v|^2 \cos^2 \frac{\theta_1 }{2} +|v_*|^2 \sin^2 \frac{\theta_1}{2}) \\
& \qquad \qquad \qquad - 
 \Psi(|v_*|^2 \cos^2 \frac{\theta_1}{2} +|v|^2 \sin^2 \frac{\theta_1}{2}) \Big ], 
\end{align*} 
which shows that 
there exists a $C_2>0$ independent of $n$ such that
\begin{align}\label{convex}
H_n(v,v_*) \ge C_2 \Big(\la v\ra^{2+\kappa} {\bf 1}_{\{\la v\ra/2\ge \la v_* \ra \}} + \la v_* \ra^{2+\kappa}
 {\bf 1}_{\{\la v_*\ra /2\ge \la v \ra \}}\Big) ,
\end{align} 
because, for $x_1 = \cos^2 \theta_1/2$ and $X = \la v\ra^2/(\la v\ra^2 + \la v_*\ra^2)$, 
\begin{align*}
 &x_1 \Psi(|v|^2) +(1-x_1)\Psi(|v_*|)^2 -  \Psi( x_1 |v|^2 +(1-x_1)|v_*|^2 )
 =\big(\la v\ra^2 + \la v_*\ra^2\big)^{1+\kappa/2}\\
& \times  \Big\{ 
x_1  X^{1+\kappa/2}  + (1-x_1) (1-X)^{1+\kappa/2}   - \Big ( x_1 X  +(1-x_1) (1-X) \Big ) ^{1+\kappa/2} \Big\}
\end{align*}
and $4/5 \le X \le 1$ if $\la v \ra \ge 2 \la v_* \ra$. 

Note that $\int \la v \ra^2 dF_t^n(v) = \int \la v \ra^2 F_0(v)$ and $\int_{|v| \le R} dF_t^n(v) \ge 1/2$ for a sufficiently large $R>0$ 
independent of $n$. 
Integrating the equation \eqref{equation-345}, we have
\begin{align}\label{imp-567}
&\int_{\RR^3} \la v \ra^{2+\kappa}  dF^n_t(v ) + C'_2  \int_0^t \int_{\RR^3_v}
\la v \ra^{2+\kappa} \min \{\la v \ra^\gamma, n^\gamma\} dF^n_{\tau}(v) d \tau \\
&\le \int_{\RR^3} \la v \ra^{2+\kappa}  dF_0(v ) \notag  \\
& \qquad + C'_3\Big ( \int_{\RR^3} \la v \ra^2 dF_0(v) \Big)\int_0^t \int_{\RR^3_v}
\la v \ra^{\max\{2, \kappa\}} \min \{\la v \ra^\gamma, n^\gamma\} dF^n_{\tau}(v) d \tau, \notag
\end{align}
for suitable positive constants $C'_2, C'_3$ independent of $n$, which yields \eqref{uniform-hard-strong}
since, by taking $R_0$ such that  $C'_2 R_0^{\min\{2, \kappa \}} \ge 2 C'_3 \int_{\RR^3} \la v \ra^2 dF_0(v) $,  we have
\begin{align*}
C'_2  \int_{|v| \ge R_0}
&\la v \ra^{2+\kappa} \min \{\la v \ra^\gamma, n^\gamma\} dF^n_{\tau}(v) \\
&\qquad \ge 2C'_3 \Big(\int_{\RR^3} \la v \ra^2 dF_0(v) \Big) \int_{|v| \ge R_0}
\la v \ra^{\max\{2, \kappa\}} \min \{\la v \ra^\gamma, n^\gamma\} dF^n_{\tau}(v) \,.
\end{align*}

Thanks to \eqref{uniform-hard-strong},  we obtain
\begin{align}\label{limit-hard-strong}
\sup_{t \in [0,T] } \int_{\RR^3_v} \la v \ra^{2+\kappa} dF_t(v) + \int_0^T \int_{\RR^3_v}
\la v \ra^{2+\kappa+ \gamma} dF_{\tau}(v) d\tau \le C_T,
\end{align}
by taking the limit. 
By using \eqref{uniform-hard-strong} and \eqref{limit-hard-strong} we can show 
that $F_t$ is a weak measure valued solution in the sense of 
\eqref{definition2-2} when the initial datum satisfies \eqref{high-moment}, similar as in the preceding section.

\bigskip
\subsection{Case: 
the initial  datum $F_0 \in P_2(\RR^3)$. }

We recall Proposition A1 of  \cite{mischler-wennberg}.

\begin{prop}\label{adding-weight}
Let $F$ belong to $P_2(\RR^3)$. Then there exists a concave function $\rho(r) \in C^2$, depending on $F$, such that 
$
\rho(r) \rightarrow \infty$ as $r \rightarrow \infty$, $r \rho(r)$ is convex, 
\begin{align}\label{prime-est}
r  \rho'(r) \lesssim \log r,
\end{align}
 and such that
for all $\varepsilon >0$ and $\alpha \in (0,1)$, $\big(\rho(r) - \rho(\alpha r)\big )r^\varepsilon \rightarrow \infty$ as 
$r \rightarrow \infty$, and such that
\begin{align}\label{plus-moment}
\int \la v \ra^2 \rho(|v|^2) dF(v) < \infty. 
\end{align}

\end{prop}
\begin{proof}
The proposition is the same as Proposition A1 of \cite{mischler-wennberg}, except for the fact that $F$ is a probability measure
and \eqref{prime-est} is explicitly stated. Without loss of generality we may assume that
$\int_{\RR^3} \la v \ra^2 dF(v) =2$. Take $0 = r_0 < r_1 < \cdots < r_j < \cdots$ such that
\[ \int_{|v| \ge r_j} \la v \ra^2 dF(v) \le  2^{-(j-1)}. \]
One  can take $\{r_j \}_{j=1}^\infty$ such that $r_j > e^j$ and $r_{j+1} - r_j > r_j - r_{j-1}$. Let $\rho_1(r)$ be linear
in each $[r_j, r_{j+1})$ and such that $\rho_1(r_j) = j$. It is obvious that
$\rho_1(r)$ is concave, increasing to infinity with $r$, and 
\[
\int_{\RR^3} \la v \ra^2 \rho_1(|v|^2) dF(v)  \le \sum_{j=1}^\infty  2^{-(j-1)}( j +1)< \infty.
\]
Finally we put 
\[
\rho(r) = \frac{1}{r}\int_0^r \frac{1}{y} \int_0^y \Big( \rho_1(z) +1 -\big( \log(e+ z) \big)^{-1} \Big) dz dy
\]
in order to make $r \rho(r)$ convex. Since $\rho(r) \le \rho_1(r) +1$, we have \eqref{plus-moment}, and moreover
\eqref{prime-est} follows from $\rho_1(e^j) \le j$. The rest of the proof is quite the same as the one of Proposition 
A1 of \cite{mischler-wennberg}.
\end{proof}

Applying this proposition to the initial datum $F_0 \in P_2(\RR^3)$, we have $$\int \la v \ra^2 \rho(|v|^2) dF_0(v) < \infty,$$
and moreover for each fixed $n$ and $T >0$
\[
\sup_{0 \le t \le T} \int \la v \ra^2 \rho(|v|^2) dF_t^n (v) < \infty,
\]
by considering $W_\delta(v) = \la v \ra^2 \rho(|v|^2) (\la \delta v \ra^2 \rho(\delta^2|v|^2) )^{-1}$as in the proof of 
Theorem \ref{existence-cutoff-BE} because it follows from the concavity of $\rho(r)$ that $\rho(2r) + \rho(0) \le 2\rho(r)$,
so that 
\[
W_\delta(v') \lesssim W_\delta(v) + W_\delta(v_*).
\]

We follow the proof of Lemma 2.2 iii) of \cite{mischler-wennberg}, in noticing that our $b(\cos \theta)$ is unbounded.
Replace $\Psi_\kappa(x)$ in the preceding procedure by $\Psi(x) = x \rho(x)$. Then
\begin{align*}
\Psi(|v'|^2) &= \Big (Y(\theta) + Z(\theta)\cos \varphi \Big ) \rho(Y(\theta) + Z(\theta)\cos \varphi)\\
& \le \Big (Y(\theta) + Z(\theta)\cos \varphi \Big )\Big( \rho(Y(\theta)) + \rho'(Y(\theta) )Z(\theta) \cos \varphi \Big),
\end{align*}
because the first factor is non-negative and $\rho$ is concave. Therefore,  we have 
\begin{align*}
\int_0^\pi \Psi(|v'|^2) d \varphi 
\le \pi \Psi (Y(\theta))   + \frac{\pi}{2}Z(\theta)^2 \rho'(Y(\theta)),
\end{align*}
instead of \eqref{recall}. Similarly we have 
\begin{align*}
\int_0^\pi \Psi(|v'_*|^2) d \varphi 
\le \pi \Psi (Y(\pi-\theta))   + \frac{\pi}{2}Z(\theta)^2 \rho'(Y(\pi- \theta)). 
\end{align*}
Since $|Z| \le Y$,  it follows from \eqref{prime-est} that 
\begin{align}\label{G-n}
G_n(v,v_*) &\leq C_0 |v|^{2-2\varepsilon} |v_*|^{2-2\varepsilon}\int_0^\pi  \left(\frac{|Z|}{Y}\right)^{2 \varepsilon} b_n(\cos \theta)
\sin^{3-2\varepsilon} \theta  d \theta\\
& \le C_1  |v|^{2-2\varepsilon} |v_*|^{2-2\varepsilon}.\notag
\end{align}
Notice again that there exist 
 $[\theta_1, \theta_2] \subset (0, \pi/2)$ and $c_0>0$ independent of $n$ such that
$ b_n(\cos \theta)\sin \theta  \ge c_0 $ on $[\theta_1, \theta_2] $, and $F(x) = 
\Psi(A) +\Psi(B) -\Psi(Ax + B(1-x)) - \Psi(A(1-x)+Bx)$ 
takes the maximum at $x =1/2$ and is decreasing with respect to
$|x-1/2|$.
Therefore, if $\tau = \sin^2 \theta_1/2 \in (0,1)$  and if $|v_*| \ge 2 |v|$, we have 
\begin{align*}
&H_n(v,v_*) \ge  2\pi \int_{\theta_1}^{\theta_2} b_n(\cos \theta)\sin \theta 
 \Big [\Psi(|v_*|^2 )+ 
  \Psi(|v|^2)     \\
& \quad   -  \Psi(|v|^2 \cos^2 \frac{\theta}{2} +|v_*|^2 \sin^2 \frac{\theta}{2}) -
\Psi(|v_*|^2 \cos^2 \frac{\theta}{2} +|v|^2 \sin^2 \frac{\theta}{2}) \Big ]d\theta \\
&\ge 2 \pi (\theta_2 -\theta_1) c_0 \Big[ |v|^2\rho(|v|^2) + |v_*|^2\rho(|v_*|^2)\\
& \quad -
\Big((1-\tau)|v|^2 + \tau |v_*|^2\Big)\rho(((1-\tau)|v|^2 + \tau |v_*|^2)\\
&\quad -
\Big((1-\tau)|v_*|^2 + \tau |v|^2\Big)\rho(((1-\tau)|v_*|^2 + \tau |v|^2)\Big]\\
&\ge  2 \pi (\theta_2 -\theta_1) c_0 \Big[|v_*|^2 \rho(|v_*|^2) -(|v|^2 + |v_*|^2) \rho( \a |v_*|^2 ) \Big], 
\end{align*}
where $\a = \max\{ \frac{1+3\tau}{4}, 1- \frac{3\tau }{4}\}$, {because 
$\rho$ is increasing and  $(1-\tau)|v|^2 + \tau |v_*|^2 \le (1+3\tau)|v_*|^2/4$.}
Since $|v|^2 \rho(|v_*|^2) \lesssim |v|^{2-2\varepsilon}
|v_*|^{2-2\varepsilon}$ when $|v_*| \ge 2 |v|$ and since the similar result holds when  $v$ and $v_*$ are exchanged,
it  follows from Proposition \ref{adding-weight} that there exist constants $C_2, C_3 >0$ independent of $n$ such that
\begin{align}\label{convex-2}
C_2 \Big( |v|^{2-\varepsilon} {\bf 1}_{\{|v| /2\ge |v_*|\}} + |v_* |^{2-\varepsilon} {\bf 1}_{\{|v_*| /2\ge |v|\}}\Big) \le H_n(v,v_*)
+ C_3 |v|^{2-2\varepsilon}
|v_*|^{2-2\varepsilon},
\end{align} 
which, together with \eqref{G-n},  concludes 
\begin{align}\label{limit-hard-strong-bis-n}
\sup_{t \in [0,T] } \int_{\RR^3_v} \la v \ra^{2} \rho(|v|^2)dF_t^n(v) + \int_0^T \int_{\RR^3_v}
\la v \ra^{2-\varepsilon} \min\{ \la v\ra^\gamma, n^\gamma\}dF^n _{\tau}(v) d\tau \le C_T,
\end{align}
and by taking $n \rightarrow \infty$, we get 
\begin{align}\label{limit-hard-strong-bis}
\sup_{t \in [0,T] } \int_{\RR^3_v} \la v \ra^{2} \rho(|v|^2)dF_t(v) + \int_0^T \int_{\RR^3_v}
\la v \ra^{2-\varepsilon + \gamma} dF_{\tau}(v) d\tau \le C_T,
\end{align}
similar to \eqref{limit-hard-strong}. As proved in the preceding section, the above two estimates
show
the 
existence of the weak solution for the initial datum $F_0 \in P_2(\RR^3)$ in the  hard potential  case.
 
The proof of the energy conservation law \eqref{definition2-1} for the hard potentials is due to the reverse Povzner
inequality, as stated in p.293 of \cite{villani}. If one considers $\Psi_\kappa (x) =(1+x)^{1+\kappa/2}$ with $\kappa <0$, then 
it follows from the concavity of $\Psi_\kappa$ that  $-H(v,v_*) \ge 0$ and 
\[
G(v,v_*) \le O(|\kappa|) |v|^2|v_*|^2,
\]
so that
\[
\int \la v \ra^{2+\kappa} dF_t(v) \ge \int \la v \ra^{2+\kappa} dF_0(v) - O(|\kappa|)\Big(\int |v|^2 dF_0(v)\Big)^2,
\]
where $H, G$ are defined in the subsection \ref{5-1} with $b_n$ there replaced by $b$. 
Letting $\kappa \rightarrow -0$, we get \eqref{definition2-1}, since the opposite inequality is obvious.

Finally, we remark the existence of a weak solution with moments gain of any order at a fixed  $T>0$   for 
the initial datum $F_0 \in P_2(\RR^3)$. Indeed, by \eqref{limit-hard-strong-bis} we see that
\[
\int_{\RR^3_v}
\la v \ra^{2-\varepsilon + \gamma} dF_{\tau}(v)  < \infty
\]
for almost all $\tau \in [0,T]$. Choose a $t_1 >0$ arbitrarily close to $0$ as an initial time. Then we get 
\[
\int_{t_1}^T \int_{\RR^3_v}
\la v \ra^{2-\varepsilon+ 2\gamma} dF_{\tau}(v) d\tau < \infty.
\]
We can repeat the same procedure again, and so on,such  that
\[
0< t_1 < t_2 < \cdots < t_n < \cdots <T.
\]
Using the method in the subsection \ref{5-1} (propagation of the moments), we see that
for any $\ell >0$
\[
\int \la v \ra^\ell dF_T(v) < \infty.
\]

\bigskip
\section{Smoothing effect of the singular cross section}\label{S5} 
\setcounter{equation}{0}

{We start the following variant of Theorem 1.2 in Fournier \cite{Fournier-2015} . 
\begin{prop}\label{fournier}
Let $B(\cdot)$ satisfy \eqref{cross-section}-\eqref{cross-section2} and let $-2\le \gamma \le 2$.
If $\gamma \le 0$ and if $F_0 \in P_2(\RR^3)$ is not a single Dirac mass, then for any weak solution in the sense of Definition
\ref{weak-solution} to the Cauchy problem \eqref{BE}-\eqref{initial} we have 
\[ 
\mbox{\rm supp}\, \, F_t = \RR^3 \enskip \mbox{for all} \enskip  t >0\,.
\]
When $\gamma >0$, if, for $\kappa > \gamma$,  $F_t$ satisfies 
\begin{align*}
\forall T >0, \enskip \exists C_T >0 \enskip \mbox{such that} \enskip 
\sup_{t \in [0,T] } \int_{\RR^3_v} \la v \ra^{2+\kappa} dF_t(v)  \le C_T
\end{align*}
then we have the same conclusion.
\end{prop}

\begin{proof}
We first consider the case $\gamma \le 0$.
For $\psi(v) \in C_b^2(\RR^3)$ we see that 
the map 
$$t \mapsto   \int_{\bR^3}\int_{\bR^3}\int_{\bS^2}b(\cdot)|v-v_*|^\g (\psi(v_*')+\psi(v')
		-\psi(v_*)-\psi(v))dF_t(v)dF_t(v_*)d\s 
$$
belongs to $C([0,\infty))$ under the assumption  $F_0 \in P_2(\RR^3)$. 
In fact, since $2 + \gamma \le 2$,  it follows from \eqref{use-234} and \eqref{justification} that,
for any $\eta >0$ there exists $R >0$  independent of $t$ such that 
\begin{align*}
&\Big|
\iint_{\{|v| \ge R\} \cup \{|v_*| \ge R\}}  \int_{\bS^2} 
b(\cdot)|v-v_*|^\g (\psi'_* +\psi'
		-\psi_* -\psi )dF_t(v)dF_t(v_*)d\s \Big|  < \eta.
\end{align*}
Therefore, the similar way to \eqref{weak-convergence-product} leads us to the desired continuity, from which and \eqref{definition2-2}
it follows that the map $t \mapsto \int_{\bR^3}\psi(v)dF_t$ belongs to $C^1([0, \infty))$.
When $\gamma >0$ we have the same $C^1$ regularity under the above additional condition. 

The rest of the proof is  the similar as
the one of  \cite[Theorem 1.2]{Fournier-2015}. Indeed, 
it follows from conservations of momentum and energy in Definition
\ref{weak-solution} that 
\[
\int_{\RR^3} |v-v_0|^2 dF_t(v) = \int_{\RR^3} |v-v_0|^2 dF_0(v) >0, \enskip \forall v_0 \in \RR^3,
\] 
which just implies Step 1 in the proof of \cite[Theorem 1.2]{Fournier-2015}, that is, 
$F_t$ is not a single Dirac mass for all $t \ge 0$.  
Let  $B(v_0, \varepsilon)$ be an open ball with a radius $\varepsilon >0$ centered at a 
$v_0 \in \RR^3$ and  suppose that 
$\int_{B(v_0, \varepsilon)} dF_t(v) =0 
$ for a $t >0$.  Then, by \eqref{definition2-2} we have,  
for any $0 \le \psi \in C_0^2(B(v_0, \varepsilon))$, 
\begin{align*}
\iint_{\bR^3 \times \bR^3}\int_{\bS^2}&b(\cdot)|v-v_*|^\g (\psi(v_*')+\psi(v')
		)dF_t(v)dF_t(v_*)d\s  \\
&= 2 \pa_t \big(\int_{\bR^3}\psi(v)dF_t(v)\Big) =0\,,
\end{align*}
because  the map 
$s \mapsto \int_{\bR^3}\psi(v)dF_s$ belongs to $C^1([0, \infty))$ and $ \int_{\bR^3}\psi(v)dF_s \ge 0$.
Since $b >0$ in $(0, \pi/2)$ and $0 \le \psi \in C_0^2(B(v_0, \varepsilon))$ can be taken arbitrarily, we see that
\begin{align*}
\iint_{\bR^3 \times \bR^3}&\int_{\bS^2}{\bf 1}_{\{v'(v,v_*, \sigma) \in B(v_0, \varepsilon)\}}\\
&\qquad \times {\bf 1}_{\{v \ne v_*, 0<  (v-v_*) \cdot \s <|v-v_*| \}}
dF_t(v)dF_t(v_*)d\s =0.
\end{align*}
Thus, Step 2 is also similar. Other Steps are quite the same as in the proof of \cite[Theorem 1.2]{Fournier-2015}.
\end{proof}

}

The following coercivity estimate estimate is a key for the proof of Theorem \ref{smoothing-1st-try}.

\begin{prop}\label{first-coercivity}
Let $B(\cdot)$ satisfy \eqref{cross-section}-\eqref{cross-section2} and assume  $2 \ge \gamma > -2s$.
{Let $F_0 \in P_2(\bR^3)$ be not equal to a single Dirac mass. For any $t_0 >0$ there exists a 
weak solution 
$F_t \in C([0,\infty);P_2(\bR^3))$ to the Cauchy problem \eqref{BE}-\eqref{initial} 
satisfying the following;
there exist $T>0$ and $c_0, C >0$ such that for any $f \in \cS(\RR^3)$ and any $t \in [t_0,t_0 +T]$
\begin{align}\label{coer-non-deg}
-\Big( Q(F_t, f), f \Big) &= \iiint
\Phi\,b \,\,  ( f^2 -f'f) dv d\sigma dF_t(v_*)\\
\notag &\ge c_0 \|\la v \ra^{\gamma/2} f\|^2_{H^s} -C \|\la v \ra^{\gamma/2} f\|^2_{L^2}\,.
\end{align}
}
\end{prop}

\begin{proof}
{When $\gamma >0$, in view of the moment gain property,  we may assume $F_t$ satisfies \eqref{limit-hard-strong}
with an interval $[0, T]$ replaced by $[t_0/2, T+t_0]$ with $\kappa >\gamma$. 
It follows from Proposition \ref{fournier} that 
\begin{align}\label{only-this}
\mbox{\rm supp}\, \, F_{t_0}  = \RR^3 \,.
\end{align}
For the sake of the simplicity of the notation, we write $0$ instead of $t_0$, in what follows.
}

For $G \in P_2$, put
$$
\cC_\gamma(G,\, f) = \iiint_{\RR^3 \times \RR^3 \times \SS^2}  b(.) | v-v_*|^\gamma  (f'-f)^2dv d\sigma dG(v_*),
$$
and
note that
$$
\big(Q(G,\, f),\,f\big) =-\frac{1}{2} \cC_\gamma(G,\, f)
+ \frac{1}{2}\iiint \Phi\,b \,\,  ( f'^2 -f^2) dv d\sigma dG(v_*)\,.
$$
It follows {}from the Cancellation Lemma and Remark 6 in \cite{ADVW}
that
\begin{align}\label{remainder-213}
&\left|\iiint \, b |v-v_*|^\gamma \,\, (f^2 -f'^2) dv d\sigma dG(v_*)\right|
\lesssim
\left|\iint |v-v_*|^\gamma \,\, f^2   dv dG(v_*)\right| \\
&\lesssim  \int \la v_* \ra^{|\gamma|} dG(v_*) \|f\|^2_{H^{(-\gamma/2)^+}_{\gamma/2}}\,\,,\notag
\end{align}
where the last inequality  in the case $\gamma \ge 0$ is trivial. While for $\gamma <0$,
this follows {}from the fact that
\begin{align}\label{useful-1}
|v-v_*|^\gamma \lesssim \la v \ra^\gamma\{ {\bf 1}_{|v-v_*| \ge \la v \ra/2}
+ {\bf 1}_{|v-v_*| < \la v \ra/2} \la v_*\ra^{-\gamma} |v-v_*|^\gamma\},
\end{align}
and the Hardy inequality $\sup_{v_*} \int |v-v_*|^\gamma |g(v)|^2 dv \lesssim \|g\|^2_{H^{-\gamma/2}}$.

For the proof of the proposition, it now suffices to consider only the quantity $\cC_\gamma(G,\, f)$ with $G = F_t$
because one can apply the interpolation inequality
\[
\|f\|^2_{H^{(-\gamma/2)^+}_{\gamma/2}}  \le \varepsilon \|f\|^2_{H^s_{\gamma/2}} + C_\varepsilon \|f\|^2_{L^2_{\gamma/2}}
\]
to \eqref{remainder-213} in view of $\gamma >-2s$. 
{The case $\gamma=0$ is easy.  
In fact, it follows from Proposition 1
of \cite{ADVW} that
\begin{align*}
 (2\pi)^{3}\cC_0(G,\, f)  &=\iint b\Big(\frac{\xi}{|\xi|} \cdot \sigma\Big) \psi(0)(|h(\xi)|^2 + |h(\xi^+)|^2)d \sigma d \xi \\
&- 2 \mbox{Re} \iint b\Big(\frac{\xi}{|\xi|} \cdot \sigma\Big) \psi(\xi^{-})h(\xi^+)\overline{h(\xi)}d \sigma d \xi\\
&\ge \int  \Big( \int b\Big(\frac{\xi}{|\xi|} \cdot \sigma\Big) (\psi(0) -|\psi(\xi^-)| ) d\sigma \Big) |h(\xi)|^2 d\xi
\end{align*}if $\psi(\xi) = \cF(G), h(\xi) = \cF(f)$.  
Consider $G = F_t, t \in [0,T]$. 
Since $\mbox{\rm supp}\, \, F_{0}  = \RR^3$, we may apply the method in the subsection 2.1 of \cite{MY}, developed 
under the condition that the initial measure $F_0$ is not concentrated on the straight line.
Then we have the desired estimate for $0 \le t \le T$ with a sufficiently small
$T >0$, by using the continuity of $\psi(t, \xi) = \cF(F_t)$.


We now consider the case $\gamma \ne 0$, following the argument used in the proof of Proposition 2.1 of \cite{amuxy7}
(see also  the one of Lemma 2
of \cite{ADVW}).
Set  $B(R) = \{v \in \RR^3\,;\, |v| \le R\}$ for $R >0$ and  $B_0(R,r) =
\{ v \in B(R)\,; \, |v-v_0| \ge r\}$ for  a $v_0 \in \RR^3$ and $r \ge 0$.
For a moment, we concentrate to handle $C_\g(F_0, f)$.
By means of  $\mbox{\rm supp}\, \, F_{0}  = \RR^3$,   for any 
positive constants $R > 1 > r_0$  we see that
\begin{equation}\label{uni-g}
\mbox{ the positive measure 
$\chi_{B_0(R, r_0)} F_0$ is not concentrated on the straight line, 
}
\end{equation}
where $\chi_A$ denotes a characteristic function of  the set $A \subset \RR^3$. }
Let $\varphi_R(v)$ be a non-negative  smooth function not greater  than one, which is 1
for $| v| \ge 4R$ and $0$ for $| v| \le {2R}$.
In view of
\[
\frac{\la v \ra}{4} 
\le|v-v_*|\le 2 \la v \ra \enskip
\mbox{on supp $(\chi_{B(R)})_* \varphi_R$}\,,
\]
we have
\begin{align*}
&4^{|\gamma|}\Phi(|v-v_*|) (f'-f)^2
\ge \big( \chi_{B(R)}\big)_* \big(\la v \ra^{\gamma/2} \varphi_R\big)^2(f'-f)^2 \\
&\ge \big( \chi_{B(R)}\big)_* \Big[\frac{1}{2} \Big(\big(\la v \ra^{\gamma/2} \varphi_R f \big)'-
\la v \ra^{\gamma/2} \varphi_R f \Big)^2
 -\Big( \big(\la v \ra^{\gamma/2} \varphi_R\big)'
 - \la v \ra^{\gamma/2} \varphi_R  \Big)^2 {f' }^2 \Big]\,.
\end{align*}
It follows {}from the mean value theorem that for a $\tau \in (0,1)$
\begin{align*}
\left|\big(\la v \ra^{\gamma/2} \varphi_R\big)'
 - \la v \ra^{\gamma/2} \varphi_R \right| &\lesssim \la v+ \tau(v'-v)\ra^{\gamma/2-1}|v-v_*|\sin\frac{\theta}{2}\\
 &\lesssim \la v_* \ra^{ |\gamma/2-1|} \la v'-v_* \ra^{\gamma/2} \sin\frac{\theta}{2}\\
 &\lesssim \la v_* \ra^{|\gamma/2|+  |\gamma/2-1|} \la v' \ra^{\gamma/2} \sin\frac{\theta}{2},
\end{align*}
because $|v-v_*|/\sqrt 2 \le |v'-v_*| \le |v+ \tau(v'-v) -v_*| \le |v-v_*|$ for $\theta \in [0,\pi/2]$. Therefore,
we have
\begin{align}\label{coer-out}
\cC_\gamma(F_t,\, f) \ge 2^{-1-2|\gamma|} \cC_0 ( \chi_{B(R)}F_t,\, \varphi_R \la v \ra^{\gamma/2}   f     )
- C_R \int dF_t(v_*) \|f\|^2_{L^2_{\gamma/2}},
\end{align}
for a positive constant $C_R \sim R^{|\gamma|+  |\gamma-2|} $.
For a set $B(4R)$ we take a finite covering
\[
B(4R) \subset \mathop{\cup}_{v_j \in B(4R)} A_j \,, \enskip A_j = \{ v \in \RR^3\,;\, |v-v_j| \le \frac{r_0}{4}\}\,.
\]
For each $A_j$ we choose a non-negative smooth function $ \varphi_{A_j}$ which is  $1$ on $A_j$
and $0$ on $\{|v-v_j|\ge r_0/2\}$.
Note that
\[
\frac{r_0}{2} 
\le|v-v_*|\le 6R  \enskip
\mbox{on supp $(\chi_{B_j(R,r_0)})_* \varphi_{A_j}$}\,.
\]
Then we  have
\begin{align*}
&\Phi(|v-v_*|) (f'-f)^2
\gtrsim \min\{r_0^{\gamma^+}, R^{-(-\gamma)^+}\} \big( \chi_{B_j(R,r_0)}\big)_* \varphi_{A_j}^2(f'-f)^2 \\
&\gtrsim
R^{-\gamma^+} \min\{r_0^{\gamma^+}, R^{-(-\gamma)^+}\} \big(  \chi_{B_j(R,r_0)}\big)_* \\
&\quad \times
\left [\frac{1}{2} \Big(\big(\la v \ra^{\gamma/2} \varphi_{A_j} f \big)'-
\la v \ra^{\gamma/2} \varphi_{A_j} f \Big)^2
 -\Big( \big(\la v \ra^{\gamma/2} \varphi_{A_j}\big)'
 - \la v \ra^{\gamma/2} \varphi_{A_j}  \Big)^2 {f' }^2 \right]\,.
\end{align*}
Since $\left|\big(\la v \ra^{\gamma/2} \varphi_{A_j}\big)'
 - \la v \ra^{\gamma/2} \varphi_{A_j} \right| \lesssim R^{|\gamma|+1}\la v'\ra^\gamma \sin\theta/2$ if $|v_*|\le R$, we obtain
\begin{align}\label{coer-inner}
\cC_\gamma(F_t,\, f)
&\gtrsim
 \min\{(r_0/R)^{\gamma^+}, R^{-(-\gamma)^+}\}
\cC_0 ( \chi_{B_j(R,r_0)}F_t,\, \varphi_{A_j} \la v \ra^{\gamma/2}   f     )\\
& \qquad \qquad \qquad - C'_{R,r_0} \int dF_t(v_*)  \|f\|^2_{L^2_{\gamma/2}}, \notag
\end{align}
for a positive constant $C'_{R,r_0} \sim R^{2+2|\gamma|}$.

Notice that we may replace $\chi_{B(R)}$ and $\chi_{B_j(R,r_0)}$ in \eqref{coer-out} and \eqref{coer-inner}, respectively,
by $C_0^\infty$ functions $\tilde \chi_j$ which still satisfy  \eqref{uni-g}.  Then, 
from the argument in the subsection 2.1 of \cite{MY}
it follows that 
\[
\exists \kappa_0 >0 \enskip \mbox{such that} \enskip \widehat{ \tilde \chi_j F_0}(0) 
- \left|\widehat{ \tilde \chi_j F_0}(\xi^-)\right| \ge \kappa_0
\]
if $\xi^-$ belongs to a suitable compact set (see (2.3) of \cite{MY}).
Since for $\varphi(t,\xi) = \int e^{-iv\cdot\xi} dF_t(v)$ we have 
\[
\widehat{ \tilde \chi_j  F_t}(\xi) = \int_{\RR^3}  \varphi(t,\xi-\eta) \widehat{ \tilde \chi_j }(\eta) d\eta,
\]
it follows from the uniform continuity of $\varphi(t,\xi)$ on any fixed compact set and $|\varphi| \le 1$ that 
there exists a $T>0$ such that 
\[
 \enskip \widehat{ \tilde \chi_j  F_t}(0) 
- \left|\widehat{ \tilde \chi_j  F_t}(\xi^-)\right| \ge \frac{\kappa_0}{2} \enskip \mbox{for any $t \in [0,T]$}
\]
if $\xi^-$ belongs to the same suitable compact set stated above. 
Therefore we obtain 
\begin{align*}
\cC_0 (\tilde \chi_j F_t,\, f ) + \|f\|^2_{L^2} \gtrsim \|f\|^2_{H^s}\,,
\end{align*}
which, together with \eqref{coer-out} and \eqref{coer-inner},  shows that
there exist $c'_0, C, C' >0$ 
such that
\begin{align}\label{coer-uni-positive}
\cC_\gamma(F_t,\, f) &\ge c'_0\Big(
\|\la D \ra^s \varphi_R \la v \ra^{\gamma/2} f\|^2 + \sum_j \|\la D \ra^s \varphi_{A_j} \la v \ra^{\gamma/2} f\|^2\Big)
-C \|f\|_{L^2_{\gamma/2}}^2\\
&\ge c'_0 \|\la v \ra^{\gamma/2} f\|^2_{H^s} - C' \|f\|_{L^2_{\gamma/2}}^2, \notag
\end{align}
because $\varphi_R^2 + \sum_j \varphi_{A_j}^2 \ge 1$ and commutators
$[\la D\ra^s, \varphi_R]$, $[\la D \ra^s, \varphi_{A_j}]$ are $L^2$ bounded operators.

In the soft potential case $\gamma <0$, noting 
\begin{align*}
\Phi(|v-v_*|) (f'-f)^2
\gtrsim R^{-|\gamma|} \big( \chi_{B(R)}\big)_* (1-\varphi_{2R})^2(f'-f)^2,
\end{align*}
we can replace $\chi_{B_0(R, r_0)}$ in 
the condition  \eqref{uni-g} by $\chi_{B(R)}$, because 
one $1- \varphi_{2R}$ is enough to cover the set  $B_{4R}$, instead of finite many $\varphi_{A_j}$.
\end{proof}

\begin{proof}[Proof of Theorem \ref{smoothing-1st-try}]
As in \cite{amuxy7, MUXY-DCDS}, we use the mollifier as follows:
Let $\lambda , N_0 \in \RR$, $\delta >0$ and put
\begin{align}\label{mollifier}
M_\lambda^\delta(\xi) = \frac{\la \xi \ra^{\lambda}}{(1+ \delta \la \xi \ra )^{N_0}}\,,\, \enskip \la \xi \ra = (1+|\xi|^2)^{1/2}\,.
\end{align}
Then $M_\lambda^\delta(\xi)$ belongs to the symbol  class $S^{\lambda -N_0}_{1,0}$ of pseudo-differential operators for each fixed $\delta >0$ and
belongs to $S^{\lambda}_{1,0} $ uniformly with respect to $\delta \in (0,1]$.

We consider the case $\gamma >0$. Since we can take arbitrarily small $t_0 >0$ in both
Proposition \ref{first-coercivity} and \eqref{moment-gain} with a fixed large $\ell$, 
for the brevity we may assume
$t_0 =0$ there.  Noting $P_{\ell} (\RR^3) \subset H^{-3/2 -\varepsilon}_\ell(\RR^3)$ for $\forall \varepsilon >0$, we can apply the proof of Theorem 5.1 in \cite{amuxy7}. Namely it suffices to show,
for any $0<t_1\leq T$ that
\begin{equation}\label{l-2}
dF_t(v) = f(t,v) dv, \enskip f(t,\cdot) \in L^\infty([t_1, T]; L^2_\ell(\RR^3)).
\end{equation}
As in the proof of Theorem 4.1 of \cite{amuxy7},  we show this by  induction. 
Assume that for $0> a \ge -3/2 - \varepsilon$, we have
$$
\sup_{[0,T]} \|f(t,\cdot)\|_{H^{a}_\ell} < \infty\,.
$$
It should be noted that if $a < -3/2$ then it follows from \eqref{moment-gain} that for $t \in [0,T]$
\[
\|f(t,\cdot)\|_{H^{a}_\ell} ^2 \le \int \la \xi \ra^{2a} d\xi \left(\int (1+ |v|^{\ell})dF_t(v)\right)^2 < \infty.
\]
Put $\lambda(t) = N t +  a$ for $N >0$. When $0<s \le 1/2$, 
choose $N_0 = a +(5+\gamma)/2\ge 1- \varepsilon + (\gamma/2) >0$
such that \cite[(3.8)]{amuxy7}
is fulfilled. Put $\varepsilon_0 = (1- 2s')/8 >0$ and consider $\varepsilon = \varepsilon_0$,
where $0< s' < s$ is chosen to satisfy $\gamma +2s' >0$.
If we choose $N, T_1>0$ such that $NT_1 = \varepsilon_0$ then
\begin{equation*}\label{caution2}
s+ \lambda(T_1) - N_0 -a = s + \varepsilon_0 -N_0 \le  s-1 + 2 \varepsilon_0-(\gamma/2)
<  (s'-1/2) + 2 \varepsilon_0 <0\,,
\end{equation*}
which shows
\begin{equation}\label{H-S-est}
M_{\lambda(t)}^\delta f(t)  \in L^\infty([0, T_1];  H^s_\ell (\RR^3))\,,
\end{equation}
where 
\[
M_{\lambda(t)}^\delta f(t) = \int e^{i v\cdot \xi}M_{\lambda(t)}^\delta(\xi) \varphi(t, \xi) \frac{d\xi}{2\pi^3},
\enskip \varphi(t, \xi) =\int_{\RR^3} e^{-iv\cdot \xi} dF_t(v).
\]
When $s >1/2$ we choose $1/2 < s' < s$ such that $\gamma +2s' >0, 2s' \ge (2s-1) $.
 Choose $N_0 = a + (5+\gamma+2s'-1)/2$ such that
\cite[(3.9)]{amuxy7} is satisfied. Put $\varepsilon_0 = (\gamma +1 )/10 >0$ and consider $\varepsilon = \varepsilon_0$.
Then, we have
\begin{equation}\label{caution3}
s+ \lambda(T_1) - N_0 -a = s + \varepsilon_0 -N_0 \le  s-s' + 2 \varepsilon_0-(1+\gamma)/2
= s-s'-3 \varepsilon_0 \,. 
\end{equation}
Since we may assume $s-s' \le \varepsilon_0$, \eqref{caution3} also shows \eqref{H-S-est}.

The formula \eqref{H-S-est}, together with \cite[Theorem 3.6 and Proposition 3.8]{amuxy7},
leads us to the fact  that, for any $t\in (0, T_1]$, 
\begin{eqnarray}\label{energyequality1}
&&\frac{1}{2}\int_{\RR^3} \big|M_{\lambda(t)}^\delta f(t)\big|^2 dv
-{N}\int^t_0 \int_{\RR^3}
\Big|(\log \la D \ra)^{1/2} M_{\lambda(\tau)}^\delta f(\tau) \Big|^2dv d\tau \notag \\
&& = \frac{1}{2}\int_{\RR^3} \big|M_{\lambda(0)}^\delta f_0\big|^2  dv\\
&& + \int^t_0
\Big( Q\big (F_\tau, M_{\lambda(\tau)}^{\delta}f(\tau)\big ), \,
M_{\lambda(\tau)}^{\delta} f(\tau)\Big)_{L^2} d\tau \notag\\
&& + \int^t_0
\Big( M_{\lambda(\tau)}^{\delta}Q\big (F_\tau, F_\tau \big)
-Q\big (F_\tau, M_{\lambda(\tau)}^{\delta}f(\tau)\big ), \,
M_{\lambda(\tau)}^{\delta} f(\tau)\Big)_{L^2} d\tau, \notag
\end{eqnarray}
by the quite same manipulation as in the proof of \cite[Lemma 4.3]{amuxy7}.
Indeed, similar as \cite[(4.7)]{amuxy7}, from the definition \eqref{definition2-2}
 we start by the formula 
\begin{align*}
&\int_{\bR^3}\psi(v)dF_t- \int_{\bR^3}\psi(v)dF_{t'}=
\frac1 2 \int_{t'}^t\int_{\bR^3}\int_{\bR^3}\int_{\bS^2}b(\cdot)|v-v_*|^\g \\
&\quad \times (\psi(v_*')+\psi(v')
		-\psi(v_*)-\psi(v))dF_\tau(v)dF_\tau(v_*)d\s d\tau, \enskip
0 \le t' \le t \le T_1\,,
		\end{align*}
and set $\psi=(M_{\lambda(\bar{t})}^\delta)^2 f(\bar{t})$ with $\bar{t}=t, t'$. All other steps are
carried on by regarding $F \in P_0(\RR^3)$ as $F \in \cS'(\RR^3)$, in view of 
the pseudo-differential operator calculus.

Applying \cite[Thereom 3.6]{amuxy7} to the last term of \eqref{energyequality1}, we have 
\begin{align}\label{imp-ene}
&\frac{1}{2}\|\big (M_\lambda^\delta f \big)(t)\|_{L^2}^2
\leq  \frac{1}{2}\|f(0)\|_{H^a}^2 + \int_0^t
\Big( Q( F_\tau, \big(M_\lambda^\delta f\big)(\tau)\,) ,  \big(M_\lambda^\delta f\big)(\tau)\Big) d\tau
\notag \\
&\quad +C_F \int_0^t \Big( 
\|\big(
M_\lambda^\delta f\big)(\tau)\|_{H^{s'} _{\gamma +(2s -1)^+}} + \|f(\tau)\|_{H^a_\gamma}\Big)
\|\big(M_\lambda^\delta f\big)(\tau) \|_{H^{s'}} d\tau\\
&\quad  + C N  \int_0^t \|(\log \la D\ra)^{1/2} \big (M_\lambda^\delta f \big)(\tau)\|_{L^2}^2 d\tau\,.
\notag
\end{align}
Proposition \ref{first-coercivity}, together with the interpolation in
the Sobolev space, yields
\[
\Big( Q( F_\tau, \big(M_\lambda^\delta f\big)(\tau)\,) ,  \big(M_\lambda^\delta f\big)(\tau)\Big)
\le
-c_F \|\big(M_\lambda^\delta f\big)(\tau)\|^2_{H^s_{\gamma/2}}
+ C_F\|f(\tau)\|^2_{H_{\gamma/2}^{-2}}\,.
\]
Use an interpolation inequality
concerning weighted type Sobolev spaces with respect to variable $v$ as follows: 
For any $k\in\RR, p\in\RR_+, \eta>0$,
\begin{equation}\label{lemma2.3}
\|f\|^2_{H^{k}_p(\RR^3_v)}\leq C_\eta \|f\|_{H^{k-\eta}_{2
p}(\RR^3_v)} \| f\|_{H^{k+\eta}_0(\RR^3_v)} \enskip \mbox{, see for instance \cite{HMUY}}.
\end{equation}
Then we have, for a suitable large $\ell >0$,
\begin{equation*}
\|\big (M_\lambda^\delta f \big)(t)\|_{L^2}^2
+ c \int_0^t \|\big(M_\lambda^\delta f\big)(\tau)\|^2_{H^s_{\gamma^+/2}}d\tau
\leq  \|f(0)\|_{H^a}^2 + C\int_0^t \|f(\tau)\|^2_{H_{\ell}^{a}} d\tau\,.
\end{equation*}
Taking $\delta \rightarrow +0$ and $t=T_1$, we have $f(T_1) \in H^{\lambda(T_1)}= H^{\varepsilon_0+a}$
because $NT_1 = \varepsilon_0$. Repeating the same procedure with $0$ and $a$ replaced by
$T_1$ and $a_1 = \varepsilon_0 +a$, we attain to \eqref{l-2}, finite times,  for any small $t_1 >0$
since the above $T_1 >0$ can be arbitrarily small by choosing a large $N$. 

The case $\gamma < 0$ is similarly handled by means of \eqref{lemma2.3},
for the solution satisfying \eqref{moment-gain}.
We remark that the restriction $\gamma >-1$ comes from what we putted $\varepsilon_0 = (\gamma +1)/10$
when $s >1/2$.  

Once we obtain \eqref{l-2}, it follows from \cite[Theorem 4.1]{amuxy7} that, for any 
$0< t_0 < T $, we have 
$
f\in L^\infty([t_0, T]; \cS(\RR^3))$.
Since, by means of \cite[Proposition 2.9]{AMUXY-ARMA2011}, for any $\ell >0$ we have
\[
\|Q(f,g)\|_{L^2_\ell} \lesssim \Big(\|f\|_{L^1_{\ell+ \gamma +2s}} + \|f\|_{L^2} \Big)
\|g\|_{H^{2s}_{\ell + \gamma+2s}},
\]
the equation \eqref{BE} and the Leibniz formula lead us to 
$
F_t \in C^\infty([t_0, \infty); \mathcal{S}(\RR^3)).
$

\end{proof}
\bigskip

\noindent
{\bf Acknowledgements:} The research of the first author was supported in part
by  Grant-in-Aid for Scientific Research No.25400160,
Japan Society for the Promotion of Science. The research of the third author was
supported in part by the General Research Fund of Hong Kong, CityU No. 11303614.


\end{document}